\documentclass{amsart}
\usepackage{graphicx,psfrag,amsmath, amssymb, overpic}
\usepackage{color}

\newcommand{\Aut}{\ensuremath{\operatorname{Aut}}}
\newcommand{\Vol}{\ensuremath{\operatorname{Vol}}}

\newcommand{\Ad}{\ensuremath{\operatorname{Ad}}}
\newcommand{\St}{\ensuremath{\operatorname{St}}}
\newcommand{\Lk}{\ensuremath{\operatorname{Lk}}}
\newcommand{\F}{\mathbb{F}}

\newcommand{\R}{\mathbb{R}}
\newcommand{\Z}{\mathbb{Z}}
\newcommand{\G}{\Gamma}
\newcommand{\s}{\sigma}
\newcommand{\bs}{\backslash}
\newcommand{\forget}[1]{}
\newcommand{\quot}{\bs \! \bs}

\newcommand{\bdy}{\partial}
\newcommand{\cross}{\! \times \!}
\newcommand{\abs}[1]{{\left\vert #1 \right\vert}}

\newcommand{\f}{F} %number of faces in tessellation
\newcommand{\p}{p} %number of sides of polygons in Bourdon's building
\newcommand{\pgon}{\ensuremath{\p}\textrm{--gon}} %% shorthand for polygon with \p sides
\newcommand{\pgons}{\ensuremath{\p}\textrm{--gons}} %% shorthand for polygons with \p sides
\newcommand{\val}{v} %valence of edges in Bourdon's building i.e. thickness
\newcommand{\K}{K_{\val,\val}} %link of vertices in Bourdon's building
\newcommand{\I}{I_{\p,\val}} %Bourdon's building
\newcommand{\HH}{\mathbb{H}} % Hyperbolic space

\newcommand{\Gg}{\G_{\p,\val,g}} %shorthand for the lattice with surface quotient

\newcommand{\Ind}{\mathcal{I}} %indexing

\let\oldmarginpar\marginpar
\renewcommand\marginpar[1]{\oldmarginpar[\raggedleft\footnotesize #1]%
{\raggedright\footnotesize #1}}

\newtheorem{theorem}{Theorem}[section]
\newtheorem*{maintheorem}{Main Theorem}
\newtheorem{prop}[theorem]{Proposition}
\newtheorem{lemma}[theorem]{Lemma}
\newtheorem{corollary}[theorem]{Corollary}

\newtheorem{remark}[theorem]{Remark}

\theoremstyle{definition}
\newtheorem{definition}[theorem]{Definition}
\newtheorem{facts}[theorem]{Facts}
\newtheorem{our_example}[theorem]{Example}

\newtheorem*{namedtheorem}{\theoremname}
\newcommand{\theoremname}{testing}

\title{Surface quotients of hyperbolic buildings}

\author{David Futer}\address{Department of Mathematics, Temple University, Philadelphia PA 19122, USA}\email{dfuter@temple.edu}

\author{Anne Thomas}\address{School of Mathematics and Statistics, University of Sydney, Sydney NSW 2006, Australia}\email{anne.thomas@sydney.edu.au}\thanks{Futer is supported in part by NSF Grant No. DMS-1007221. Thomas was supported in part by NSF Grant No. DMS-0805206 and in part by EPSRC Grant No. EP/D073626/2, and is currently supported by ARC Grant No. DP110100440.}

\date{\today}

\begin{document}

\begin{abstract} Let $\I$ be Bourdon's building, the unique simply-connected $2$--complex such that all $2$--cells are regular right-angled hyperbolic $\p$--gons and the link at each vertex is the complete bipartite graph $\K$.  We investigate and mostly determine the set of triples $(\p,\val,g)$ for which there exists a uniform lattice $\Gamma = \Gamma_{\p,\val,g}$ in $\Aut(\I)$ such that $\Gamma \bs \I$ is a compact orientable surface of genus $g$.  Surprisingly, for some $\p$ and $g$ the existence of $\Gamma_{\p,\val,g}$ depends upon the value of $\val$.  The remaining cases lead to open questions in tessellations of surfaces and in number theory.  Our construction of $\Gamma_{\p,\val,g}$ as the fundamental group of a simple complex of groups, together with a theorem of Haglund, implies that for $\p \geq 6$ every uniform lattice in $\Aut(\I)$ contains a surface subgroup.  We use elementary group theory, combinatorics, algebraic topology, and number theory.  \end{abstract}

\maketitle

\section{Introduction}

Let $\I$ be Bourdon's building, the unique simply-connected $2$--complex such that all $2$--cells are regular right-angled hyperbolic $\p$--gons, and the link at each vertex is the complete bipartite graph $\K$ (see Section \ref{ss:links} below).  Bourdon's building is a hyperbolic building, in which each apartment is a copy of the hyperbolic plane tiled by regular right-angled $\p$--gons.  As is well known, there are many surface quotients of each such apartment.  In this paper, we investigate surface quotients of the entire building $\I$.

The automorphism group $G=\Aut(\I)$ may be equipped with the compact-open topology, and is then a totally disconnected, locally compact group, nondiscrete for $\val \geq 3$.  In this topology, a uniform lattice in $G$ is a subgroup $\G < G$ acting cocompactly on $\I$ with finite cell stabilizers (see Section~\ref{ss:lattices} below).  Bourdon's building and its lattices have been studied by, for example, Bourdon~\cite{B97}, Bourdon--Pajot~\cite{BP00}, Haglund~\cite{H02,H06,H08}, Kubena--Thomas~\cite{KT08}, Ledrappier--Lim~\cite{LL09}, R\'emy~\cite{R02}, Thomas~\cite{T06}, and Vdovina~\cite{V05}.

Our Main Theorem below considers surface quotients of $\I$ by the action of uniform lattices in $G$.    Before stating this  result, we note that if  for some lattice $\G$ in $\Aut(\I)$, the quotient $\G \bs \I$ is isometric to a compact orientable surface $S$, then $S$ must be tiled by regular right-angled hyperbolic $\p$--gons.  If $S$ has genus $g$, an easy calculation (see Corollary \ref{c:tiling-exists}) shows that the number of $\pgons$ in any such tessellation is
\[\f:=\frac{8(g-1)}{\p - 4}.\]
Hence a necessary condition for the existence of a lattice with quotient a genus $g$ surface is that $\f$ be a positive integer.  We say that \emph{a lattice $\Gg$ exists} if there is a lattice $\G = \Gg < \Aut(\I)$ such that $\G \bs \I$ is isometric to a compact, orientable, genus $g$ surface tiled by regular right-angled hyperbolic $\p$--gons.

\begin{maintheorem}\label{t:surface_quotient_intro} Let $\p \geq 5$, $\val \geq 2$, and $g \geq 2$ be integers, and let $\I$ be Bourdon's building.  Assume that $\f = \frac{8(g-1)}{p - 4}$ is a positive integer.
\begin{enumerate}
\item\label{i:existence} \emph{Existence of $\Gg$.}
\begin{enumerate}
\item\label{i:existence v even} If $\val \geq 2$ is even, then for all $\f$, a lattice $\Gg$ exists.
\item\label{i:existence f divisible by 4} If $\f$ is divisible by $4$, then for all integers $\val \geq 2$, a lattice $\Gg$ exists.
\item\label{i:existence f composite} If $\f$ is composite, then for infinitely many odd integers $\val \geq 3$, a lattice $\Gg$ exists.
\end{enumerate}
\item\label{i:non-existence} \emph{Non-existence of $\Gg$.} \begin{enumerate}
    \item\label{i:nonexistence f odd} If $\f$ is odd, then for infinitely many odd integers $\val \geq 3$, a lattice $\Gg$ does \emph{not} exist.
        \end{enumerate}
%\item\label{i:f divisible by 4} If $\f \equiv 0 \pmod 4$ then for all values of $\val$, a lattice $\Gg$ exists.
%\item\label{i:f is 2 mod 4} If $\f \equiv 2 \pmod 4$ then for all even values of $\val$ and for infinitely many odd values of $\val$, a lattice $\Gg$ exists.
%\item\label{i:f odd} If $\f$ is odd, then for all even values of $\val$ a lattice $\Gg$ exists, but for infinitely many odd values of $\val$, a lattice $\Gg$ does \emph{not} exist.
\end{enumerate}
\end{maintheorem}

The non-existence result \eqref{i:non-existence} of the Main Theorem came as a great surprise to the authors.  We do not know of any previous results for Bourdon's building or its lattices which depend upon the value of $\val$.

The odd values of $\val$  in  \eqref{i:existence f composite} of the Main Theorem include all multiples of $15$, and more generally all multiples of $(b+1)(b^2+1)$, where $b \geq 2$ is any even number. The odd values of $\val$  in  \eqref{i:nonexistence f odd} of the Main Theorem include all integers of the form $\val = q^n$, where $q$ is an odd prime. These particular values of $\val$ that imply existence or non-existence of $\Gg$ hint at the number theory lurking in this problem.  Indeed, as we explain in Section \ref{ss:unimodularity_equation}, we reached open questions in number theory while attempting to resolve the cases $\f \not \equiv 0 \pmod 4$ when $\val$ is odd.

In the cases where a lattice $\Gg$ does exist, we construct $\Gg$ as the fundamental group of a complex of finite groups over a tessellation of a genus $g$ surface (see Section~\ref{ss:complexes_of_groups} below for background on complexes of groups).  The $\Gg$ so obtained are a new family of uniform lattices.  In particular, they are not graph products of finite groups, as considered in \cite{B97,H08}, nor are they constructed from such graph products as in \cite{KT08}, nor are they fundamental groups of finite polyhedra as in \cite{V05}, nor do they  ``come from" tree lattices as do the lattices in \cite{T06}.  

The following corollaries of \eqref{i:existence v even} and \eqref{i:existence f divisible by 4} in the Main Theorem are immediate.

%\begin{named}{Corollary 1}
\begin{corollary}[Every Bourdon building covers a surface]\label{cor:every-building-covers}
  For all $\p \geq 5$ and all $\val \geq 2$, $\Aut(\I)$ admits a lattice whose quotient is a compact orientable hyperbolic surface.
\end{corollary}
%\end{named}

% \begin{named}{Corollary 2}
\begin{corollary}[Every genus surface is covered by a Bourdon building]\label{cor:every-surface-covered}
 For all $g \geq 2$, there is a compact orientable hyperbolic surface of genus $g$ which is the quotient of some building $\I$.
\end{corollary}
% \end{named}

Bourdon's building $\I$ is a $\mathrm{CAT}(-1)$ space, and uniform lattices in $\Aut(\I)$ are quasi-isometric to $\I$.  Hence uniform lattices in $\Aut(\I)$ are word-hyperbolic groups (see for example \cite{BH99}).  An open question of Gromov is whether every one-ended word-hyperbolic group contains a surface subgroup, that is, a subgroup isomorphic to the fundamental group of a compact orientable hyperbolic surface.  Vdovina showed that when $\p = 2k$ is even, there is a uniform lattice $\G < \Aut(\I)$ which contains the fundamental group of a genus $g = 2k - 4$ surface \cite[Theorem 3]{V05}.  More recently, Haglund proved  that for all $\p \geq 6$, all uniform lattices in $\Aut(\I)$ are commensurable up to conjugacy  \cite[Theorem 1.1]{H06}.  As we explain in Section \ref{ss:complexes_of_groups} below,  since we construct $\Gg$ as the fundamental group of a \emph{simple} complex of groups, the (topological) fundamental group of the quotient genus $g$ surface embeds in $\Gg$. Thus combining our construction of $\Gg$ with Haglund's theorem, we obtain the following 
special case of Gromov's conjecture:
%generalisation of Theorem 3 of \cite{V05}:

% \marginpar{What happens when $p=5$?}

% \begin{named}{Corollary 3} 
\begin{corollary}\label{cor:gromov-special-case}
For all $\p \geq 6$ and all $\val \geq 2$, every uniform lattice $\G < \Aut(\I)$ contains a surface subgroup.
\end{corollary}
% \end{named}

\noindent Although this is the first appearance of Corollary \ref{cor:gromov-special-case} in print, it also follows by combining \cite[Theorem 1.1]{H06} with Theorem 3.6 and Corollary 4.13 of Sang-hyun Kim's Ph.D. thesis  \cite{K07}.  In Section \ref{s:relationships}, we further discuss the relationships between our lattices $\Gg$, previous examples, and surface subgroups.  In particular, when $\val$ is even, Proposition \ref{p:rel with Gamma_0} explicitly constructs a common finite-index subgroup shared by  $\Gg$ and an important lattice $\Gamma_0$ described by Bourdon \cite{B97}.

We prove our Main Theorem and Corollary  \ref{cor:gromov-special-case} in Section \ref{s:proof} below, using results from Sections \ref{s:group theory}--\ref{s:homology}.  For the positive results in \eqref{i:existence} of the Main Theorem, we use the following homological necessary and sufficient conditions on a  tessellation.  Let $S_g$ be a surface of genus $g$, and let $Y$ be a tiling of $S_g$ by $\f$ copies of a regular right-angled hyperbolic $\pgon$. Note that at each vertex of $Y$, two (local) geodesics intersect at right angles.  In Section \ref{s:homology}, we prove:

% \begin{named}{Proposition 4} 
\begin{theorem}\label{prop:homology-summary}
Let $h_1,\ldots,h_n$ be the closed geodesics of the tessellation $Y$.
\begin{enumerate}
\item\label{i:weak intro} There are integers $c_i \neq 0$ so that $\sum c_i [h_i] = 0$ in homology if and only if for \emph{some} odd integer $v \geq 3$, there is a lattice  $\Gg$ such that $\Gg \bs \I \cong Y$.
\item\label{i:strong intro} The $h_i$ may be oriented so that  $\sum [h_i] = 0$ in homology if and only if for \emph{every} odd integer $\val \geq 3$, there is a lattice $\Gg$ such that $\Gg \bs \I \cong Y$.
    \end{enumerate}
\end{theorem}
% \end{named}

\noindent We also show, using \eqref{i:strong intro} of Theorem \ref{prop:homology-summary}, the negative result:

% \begin{named}{Corollary 5} 
\begin{corollary}\label{cor:no-lattice-v3}
Fix any $p \geq 5$, any integer  $\f \not \equiv 0 \pmod 4$, and let $Y$ be any tessellation of a surface by $\f$ copies of a regular right-angled $p$--gon. Then there does not exist a lattice $\G_{\p,3,g}$ such that $\G_{\p,3,g} \bs I_{\p,3} \cong Y$.
\end{corollary}
%\end{named}

\noindent In particular, when $\f$ is not divisible by $4$, there is no tessellation $Y$ of $S_g$ by $\f$ tiles, such that for every odd integer $\val \geq 3$ there is a lattice $\Gg$ with $\Gg \bs \I \cong Y$.  A theorem of Edmonds--Ewing--Kulkarni \cite{EEK} constructs \emph{some} tiling of $S_g$ by $\f$ faces, for any positive integer $\f$.  However, so far as we know, there is no classification of tessellations allowing us to determine in general the existence of a tiling satisfying even the weaker condition \eqref{i:weak intro} of Theorem \ref{prop:homology-summary}.  Indeed, even the enumeration of tessellations (also called \emph{maps on surfaces}) is a subject of current research; see \cite{Z97} for an introduction to the area.

To prove Theorem \ref{prop:homology-summary} and the negative results in \eqref{i:non-existence} of the Main Theorem, we introduce in Section \ref{s:indexings} the combinatorial data of an ``indexing" of a complex of  groups.  This generalizes the indexing of a graph of groups, appearing in, for example,~\cite{BK90}.  We establish several necessary conditions on indexings of complexes of groups with universal cover $\I$, which are used in the proof of Theorem \ref{prop:homology-summary}.  One of these conditions, \emph{parallel transport} (see Section \ref{ss:parallel} below), is 2--dimensional in nature and has no analog in Bass--Serre theory.  The connection between the homological conditions in Theorem \ref{prop:homology-summary} and our necessary conditions on indexings is via intersection pairings on homology (see Proposition \ref{prop:weak-homology}).

In Section \ref{s:indexings}, we also formulate a family of equations in $\val$,  called the \emph{unimodularity equations}, which must be satisfied whenever a tessellated surface arises as a quotient of $\I$.  We prove non-existence of solutions to the unimodularity equations for two infinite families of odd $\val$, and then explain why open questions in number theory mean that we cannot effectively determine all solutions. The values of $\val$ for which the unimodularity equations have no solution are exactly the ones that appear in \eqref{i:nonexistence f odd} of the Main Theorem.  
%	Roughly speaking, non-existence of solutions for $\val$ odd implies that $\val$ appears in \eqref{i:nonexistence f odd} of the Main Theorem.  

Our proofs of Theorem \ref{prop:homology-summary} and of the Main Theorem also use local group-theoretic necessary and sufficient conditions for a complex of groups to have universal cover $\I$. These conditions, which we establish in Section \ref{s:group theory}, generalize results in Chapter 4 of Martin Jones' Ph.D. thesis \cite{J09}.

\subsection*{Acknowledgements}  The two authors met and began talking mathematics at the Mathematical Sciences Research Institute in Fall 2007.  Our collaboration on this project could not have occurred without the encouragement of Moon Duchin, for which we are very grateful.  We thank Michael Broshi and Roger Heath-Brown for answering our number theory questions, Martin Jones and Sang-hyun Kim for explaining their thesis work,   and Martin Bridson, Norman Do, Tadeusz Januszkiewicz, and Genevieve Walsh for helpful discussions.  We are also grateful to an anonymous referee for suggesting that we include the discussion in Section \ref{s:relationships}.  The second author additionally thanks the London Mathematical Society for supporting travel during which she met Martin Jones, Sarah Rees and Guyan Robertson for access to Jones' thesis, and Donald Cartwright and James Parkinson at the University of Sydney for their hospitality in mid-2009.

\section{Background}\label{s:background}

In Section~\ref{ss:links} below, we recall the definitions of a link and of Bourdon's building $\I$.  Section~\ref{ss:lattices} then recalls some basic theory of lattices and characterizes lattices in $\Aut(\I)$.  In Section~\ref{ss:complexes_of_groups} we sketch the theory of complexes of groups needed for our constructions of lattices in $\Aut(\I)$. The equivalence between lattices and complexes of groups is summarized in Corollary \ref{c:lattice}.

\subsection{Links and Bourdon's building}\label{ss:links}

Let $X$ be a polygonal complex.  The \emph{link} of a vertex $\s$ of $X$, denoted $\Lk(\s,X)$, is the graph obtained by intersecting $X$ with a $2$--sphere of sufficiently small radius centered at $\s$.  Equivalently, $\Lk(\s,X)$ is the graph 
whose vertices correspond to endpoints of $1$--cells of $X$ that are incident to $\s$, and whose edges correspond to corners of  $2$--cells of $X$ incident to $\s$. 
% DF: rephrase. 1-cells and 2-cells may meet a vertex multiple times, due to topology. So I'm emphasizing that this
% is only *local* adjacency.
% ---------------
%	with vertices the $1$--cells of $X$ incident to $\s$ and edges the $2$--cells of $X$ 
%	incident to $\s$; two vertices in the link are joined by an edge in the link if the 
%	corresponding $1$--cells in $X$ are contained in a common $2$--cell.  
The link may be metrized by giving each edge in $\Lk(\s,X)$ length equal to the angle at $\s$ in the corresponding $2$--cell of $X$.

By definition, Bourdon's building $\I$ is a simply-connected polygonal complex with all links the complete bipartite graph $\K$ and all $2$--cells regular right-angled hyperbolic $\p$--gons, where $\val \geq 2$ and $\p \geq 5$.  Moreover, $\I$ is the \emph{unique} simply-connected polygonal complex having these links and $2$--cells \cite[Proposition 2.2.1]{B97}.

\subsection{Lattices for Bourdon's building}\label{ss:lattices}

Let $G$ be a locally compact topological group.  A discrete subgroup $\Gamma \leq G$ is a \emph{lattice} if $\Gamma\backslash G$ carries a finite $G$--invariant measure, and is \emph{uniform} or \emph{cocompact} if $\G \bs G$ is compact.
Let $S$ be a left $G$--set such that for every $s \in S$, the stabilizer $G_s$ is compact and open. Then $\G \leq G$ is discrete if and only if the stabilizers $\Gamma_s$ are finite.  The
\emph{$S$--covolume} of a discrete subgroup $\Gamma \leq G$ is defined to be
\[ \Vol(\Gamma \quot S) :=  \sum \frac{1}{|\Gamma_s|} \leq \infty
\]
where the sum is over the elements $s \in S$ belonging to some fixed fundamental domain for $\G$.

Now let $G=\Aut(\I)$ be the automorphism group of Bourdon's building $\I$, that is, the set of cellular isometries of $\I$.  When equipped with the compact-open topology, the group $G = \Aut(\I)$
is naturally locally compact, and a subgroup $\G \leq G$ is discrete if and only if no sequence of elements of $\G$ converges uniformly on compact subsets of $\I$.  Since the stabilizers in $G$ of cells of $\I$ are compact and open,  we may take the set $S$ above to be the set of cells of $\I$.
Then by the same arguments as for tree
lattices \cite[Chapter 1]{BL01}, it can be shown that a discrete subgroup $\G \leq G$ is a lattice if
and only if its $S$--covolume converges, and $\G$ is uniform if and
only if the sum above has finitely many terms, equivalently if $\G \bs S$ is compact.  Hence a uniform lattice in $G = \Aut(\I)$ is precisely a subgroup $\Gamma < G$ which acts cocompactly on $\I$ with finite cell stabilizers.

\subsection{Complexes of groups}\label{ss:complexes_of_groups}

We now sketch the theory of complexes of groups over polygonal complexes and apply this theory to the construction of lattices in $\Aut(\I)$ (see, in particular, Corollary~\ref{c:lattice} below).  We refer the reader to Bridson--Haefliger~\cite{BH99} for details.

Throughout this paper, if $Y$ is a polygonal complex, such as a tessellated surface, then $Y'$ will denote the first barycentric subdivision of $Y$, with vertex set $V(Y')$ and edge set $E(Y')$.  
Each $a \in E(Y')$ corresponds to cells
$\tau \subset \sigma$ of $Y$, and so may be oriented from $i(a)=\sigma$ to $t(a)=\tau$.  Two edges $a$ and $b$ of
$Y'$ are \emph{composable} if $i(a) = t(b)$, in which case there exists an edge $c=ab$ of $Y'$ such that $i(c) =
i(b)$, $t(c) = t(a)$,
 and $a$, $b$ and $c$ form the boundary of a triangle in $Y'$.

\begin{definition} A \emph{complex of groups} $G(Y)=(G_\sigma, \psi_a, g_{a,b})$ over a polygonal complex $Y$ is given by:
\begin{enumerate} \item a group $G_\sigma$ for each $\sigma \in V(Y')$, called the \emph{local group} at $\sigma$;
\item a monomorphism $\psi_a\colon G_{i(a)}\rightarrow G_{t(a)}$ for each $a \in E(Y')$; and \item for each pair of
composable edges $a$, $b$ in $Y'$, an element $g_{a,b} \in G_{t(a)}$, such that \[ \Ad(g_{a,b})\circ\psi_{ab} =
\psi_a \circ\psi_b \] where $\Ad(g_{a,b})$ is conjugation by $g_{a,b}$ in $G_{t(a)}$.
\end{enumerate}
\end{definition}

\noindent We will usually refer to local groups as face, edge, and vertex groups.  All of the complexes of groups we construct will be \emph{simple}, meaning that each of the $g_{a,b}$ is trivial.  In this case, inclusions of cells in $Y$ give opposite inclusions of local groups.  A simple complex of groups is also sometimes called an \emph{amalgam}.

If $G$ is a group of automorphisms of a simply-connected polygonal complex $X$, such as $X = \I$, then $G$ is said to act \emph{without inversions} on $X$ if whenever $g \in G$ fixes a cell $\s$ of $X$ setwise, $g$ fixes $\s$ pointwise.  The action of $G$ then induces a complex of groups over $Y = G \backslash X$, as follows. Let $p\colon X \rightarrow Y$ be the natural projection. For each $\sigma
\in V(Y')$, choose $\tilde\sigma \in V(X')$ such that $p(\tilde\sigma) = \sigma$. The local group $G_\sigma$ is then defined to be the stabilizer of $\tilde\sigma$ in $G$, and the monomorphisms $\psi_a$ and group elements $g_{a,b}$ are defined using further choices. The
resulting complex of groups $G(Y)$ is unique (up to isomorphism).

Let $G(Y)$ be a complex of groups over a polygonal complex $Y$.  The {\em fundamental group of $G(Y)$} is denoted by $\pi_1\bigl(G(Y)\bigr)$.  In order to give a presentation for $\pi_1\bigl(G(Y)\bigr)$, let $E^\pm(Y') := \{ a^+, a^- \mid a \in E(Y') \}$, and fix a maximal tree $T$ in the $1$--skeleton of $Y'$.  Then by Theorem 3.7 of \cite[Section III.$\mathcal{C}$]{BH99}, the fundamental group $\pi_1\bigl(G(Y)\bigr)$ is (canonically isomorphic to) the group generated by the set 
\[ \coprod_{\sigma \in V(Y')} G_\sigma \, \coprod E^\pm(Y')\]
subject to the relations:
\begin{enumerate}
\item the relations in the local groups $G_\sigma$; 
\item $(a^+)^{-1} = a^-$ and $(a^-)^{-1} = a^+$ for all $a \in E(Y')$;
\item $a^+ b^+ = g_{a,b}(ab)^+$ for all composable pairs of edges $(a,b)$;
\item $\psi_a(g) = a^+ g a^-$ for all $g \in G_{i(a)}$; and
\item $a^+ = 1$ for all $a \in T$.
\end{enumerate}

We will use the following result to prove Corollary \ref{cor:gromov-special-case} of the introduction.

\begin{prop}[Example  III.$\mathcal{C}$.3.11(1),~\cite{BH99}]\label{p:surface_subgroups} If $G(Y)$ is simple, then the (topological) fundamental group $\pi_1(Y)$ embeds in the fundamental group of the complex of groups $\pi_1\bigl(G(Y)\bigr)$.
\end{prop}

\begin{proof}  Use the presentation of $\pi_1\bigl(G(Y)\bigr)$ given above.  Since $G(Y)$ is simple, the subgroup generated by the elements $a^+$, $a \in E(Y')$, is isomorphic to $\pi_1(Y)$.
\end{proof}

The \emph{universal cover} of $G(Y)$, denoted $\widetilde{G(Y)}$, is a simply-connected polygonal complex, equipped with a canonical action of $\pi_1(G(Y))$ without inversions.  The quotient of $\widetilde{G(Y)}$ by this action is naturally isomorphic to $Y$, and for each cell $\sigma$ of $Y$ the stabilizer in $\pi_1(G(Y))$ of any lift $\widetilde{\sigma}\subset\widetilde{G(Y)}$
is a homomorphic image of $G_\sigma$. The complex of groups $G(Y)$ is called {\em developable}
if every such homomorphism $G_\sigma\to {\rm Stab}_{\pi_1(G(Y))}(\tilde\sigma)$ is injective.
Equivalently, a complex of groups is developable if it is isomorphic to the complex of groups associated as above to an action without inversions on a simply-connected polygonal complex.  Unlike graphs of groups, complexes of groups are not in general developable.

We now describe a local condition for developability.  Let $Y$ be a connected polygonal complex
and let $\sigma \in
V(Y')$.  The \emph{star} of $\sigma$, written $\St(\sigma)$, is the union of the interiors of the simplices in
$Y'$ which meet $\sigma$.  If $G(Y)$ is a complex of groups over $Y$ then, even if $G(Y)$ is not developable, each
$\sigma \in V(Y')$ has a \emph{local development}. That is, we may associate to $\sigma$ an action of $G_\sigma$
on the star $\St(\tilde\sigma)$ of a vertex $\tilde\sigma$ in some simplicial complex, such that $\St(\sigma)$ is
the quotient of $\St(\tilde\sigma)$ by the action of $G_\sigma$.

To determine the local development, its link may be computed.  We recall this construction in the case that $G(Y)$ is simple and $\s$ is a vertex of $Y$ in Section~\ref{s:group theory} below.   If $G(Y)$ is developable, then for each $\sigma \in V(Y')$, the local development  $\St(\tilde\sigma)$ is
isomorphic to the star of each lift $\tilde\sigma$ of $\sigma$ in the universal cover $\widetilde{G(Y)}$. The
local development has a metric structure induced by that of the polygonal complex $Y$.  A complex of
groups $G(Y)$ is \emph{nonpositively curved} if for all $\sigma \in V(Y')$, the star $\St(\tilde\sigma)$ is CAT$(0)$ in this induced metric.  The importance of the nonpositive curvature condition is given by:

\begin{theorem}[Bridson--Haefliger \cite{BH99}, see also Gersten--Stallings \cite{S91} and Corson \cite{C92}]\label{t:nonpos} A nonpositively curved complex of groups is developable.
\end{theorem}

Complexes of groups may be used to construct lattices, as follows.
Let $G(Y)$ be a developable complex of groups, with universal cover a locally finite polygonal complex $X$, and
fundamental group $\G$.  We say that $G(Y)$ is \emph{faithful} if the action of $\G$ on $X$ is faithful.  If $G(Y)$ is faithful, then $\G$ may be regarded as a subgroup of the locally compact group $G=\Aut(X)$.  Moreover, by the discussion in this section and in Section~\ref{ss:lattices} above, if $G(Y)$ is faithful then $\Gamma$ is discrete if and only if all local groups
of $G(Y)$ are finite, and a discrete subgroup $\G < G$ is a uniform lattice in $G$ if and only if $Y \cong \G \backslash X$ is compact.

We now specialize to the case where $Y$ is a $2$--complex in which every $2$--cell is a regular right-angled hyperbolic $\p$--gon.  Let $G(Y)$ be a complex of groups over $Y$ such that each vertex of $Y$ has local development with link $\K$.   By the Gromov Link Condition for $2$--complexes \cite[Section II.5.24]{BH99}, $G(Y)$ is nonpositively curved if and only if for each vertex $\s$ of $Y$, every injective loop in the link of the local development of $\s$ has length at least $2\pi$.  Since the $2$--cells of $Y$ are right-angled, each edge of the link $\K$ has length $\frac{\pi}{2}$.  Also, each injective loop in the graph $\K$  contains at least $4$ edges.  Thus $G(Y)$ is nonpositively curved, and so by Theorem~\ref{t:nonpos}, the complex of groups $G(Y)$ is developable.  The universal cover $\widetilde{G(Y)}$ has all vertex links $\K$ and all $2$--cells regular right-angled hyperbolic $p$--gons.  The uniqueness of Bourdon's building (see Section~\ref{ss:links} above) then implies that the universal cover $\widetilde{G(Y)}$ is isomorphic to $\I$.

For convenience, we summarize the above discussion as follows:

\begin{corollary}\label{c:lattice}  Let $Y$ be a compact $2$--complex with each $2$--cell a regular right-angled hyperbolic $\p$--gon.  Then the following are equivalent:
\begin{enumerate}
\item\label{i:complex-exists}  There is a faithful complex of finite groups $G(Y)$ such that the link of each local development of a vertex of $Y$ is $\K$. 
\item\label{i:lattice-exists} There is a uniform lattice $\G$ in $\Aut(\I)$, acting without inversions, such that $Y \cong \G \backslash \I$.
\end{enumerate}
%	Let $G(Y)$ be a faithful complex of finite groups over $Y$ such that the link of 
%	each local development of a vertex of $Y$ is $\K$.  Then the fundamental group 
%	$\G$ of $G(Y)$ is a uniform lattice in $G = \Aut(\I)$ such that $Y \cong \G \backslash \I$.
\end{corollary}

\begin{proof}For \eqref{i:complex-exists} $\Rightarrow$  \eqref{i:lattice-exists}, the lattice $\G$ is the fundamental group of $G(Y)$. For   \eqref{i:lattice-exists} $\Rightarrow$  \eqref{i:complex-exists}, each local group $G_\sigma$ is the stabilizer in $\I$ of some preimage $\tilde\sigma$ of $\sigma$.
\end{proof}

Thus, to prove the Main Theorem, it will suffice to establish the existence or non-existence of a faithful complex of finite groups $G(Y)$, where $Y$ is a tiling of a surface by regular right-angled hyperbolic $\p$--gons.

\section{Group theory and local developments}\label{s:group theory}

The main result of this section is Corollary~\ref{c:link_K}, which provides group-theoretic necessary and sufficient conditions for a simple complex of groups $G(Y)$ over a tessellation $Y$ to have local developments with links $\K$.  Corollary \ref{c:link_K} follows from Proposition \ref{p:complete_bipartite_product}, which considers more general complexes of groups $G(Y)$.  Our results in this section generalize Chapter 4 of Martin Jones' Ph.D. thesis~\cite{J09}.  In particular, in Theorem 4.17 of \cite{J09}, Jones established conditions similar to those in Proposition \ref{p:complete_bipartite_product} for $G(Y)$ a square complex of finite groups with trivial face groups.  

Let $G(Y)$ be a complex of groups, not necessarily simple, over a polygonal complex $Y$, and let $\s$ be a vertex of $Y$.  We assume that the link of $\s$ in $Y$ does not contain any loops, equivalently that none of the faces of $Y$ which are adjacent to $\s$ are glued to themselves along an edge containing $\s$. 

We recall the construction of the link in the local development $\St(\tilde\s)$ in this case (see  \cite[Definition III.$\mathcal{C}$.4.21]{BH99} for the general construction).  Suppose the vertex $\s$ has local group $G_\s = V$.  Let $\{ a_j \}$ be the set of edges of $Y'$ such that $t(a_j) = \s$ and $i(a_j)$ is the midpoint of an edge of $Y$.  Let $\{ c_k\}$ be the set of edges of $Y'$ such that $t(c_k) = \s$ and $i(c_k)$ is the barycenter of a face of $Y$.  Denote by $E_j = G_{i(a_j)}$ the corresponding edge groups and by $F_k = G_{i(c_k)}$ the corresponding face groups.  Whenever the edges $a_j$ and $c_k$ belong to the same face of $Y$, by our assumption on the link of $\s$ in $Y$ there is a unique edge $b_{kj} \in Y'$ such that $i(b_{kj}) = i(c_k)$ and $t(b_{kj}) = i(a_j)$, that is, the pair $(a_j,b_{kj})$ is composable with $a_j b_{kj} = c_k$.  By definition of a complex of groups, for each pair of composable edges $(a_j,b_{kj})$, there is a group element $g_{a_j,b_{kj}} \in V$ such that for all $g \in F_k$, 
\[ g_{a_j,b_{kj}} \, \psi_{c_k}\!(g)\, g_{a_j,b_{kj}}^{-1} = \psi_{a_j}\!\left( \psi_{b_{kj}}\!(g)\right).\]

 %Since $G(Y)$ is simple, we may identify the groups $E_j$ and $F_k$ with their images in $V$ under the monomorphisms $\psi_{a_j}$ and $\psi_{b_k}$ respectively.

\begin{definition}\label{d:link_local_development}  The \emph{link of the local development} $\St(\tilde\sigma)$ of $\s$ is the graph $L$ with:
 \begin{itemize}\item vertex set the disjoint union of the cosets $V/\psi_{a_j}\!(E_j)$; \item edge set the disjoint union of the cosets $V/\psi_{c_k}\!(F_k)$; and \item the edges between the vertices $g\psi_{a_j}\!(E_j)$ and $g'\psi_{a_{j'}}\!(E_{j'})$, where $g,g' \in V$, the cosets of $\psi_{c_k}\!(F_k)$ in \[\left(g\,\psi_{a_j}\!(E_j)\,g_{a_j,b_{kj}}\right) \cap \left(g'\,\psi_{a_{j'}}\!(E_{j'})\,g_{a_{j'},b_{kj'}}\right)\] for each $k$ such that $a_j$, $a_{j'}$, and $c_k$ belong to the same face of $Y$.\end{itemize}
\end{definition}

%In Proposition \ref{p:complete_bipartite_product} below we consider a vertex of $Y$ which is only contained in one face, then in Corollary \ref{c:link_K} below we consider $Y$ a tessellation of a surface.

\begin{prop}\label{p:complete_bipartite_product}  Suppose that $a_j$ and $a_{j'}$ belong to at least one common face of $Y$.  In the link $L$ of the local development $\St(\tilde\sigma)$:
\begin{enumerate}
\item\label{i:at least one edge} There is at least one edge connecting each vertex $g\psi_{a_j}\!(E_j)$ to each vertex $g'\psi_{a_{j'}}\!(E_{j'})$ if and only if every $h \in V$ can be written in the form \[ h = \psi_{a_j}\!(e_j) \, x \, \psi_{a_{j'}}\!(e_{j'}) \] for some $e_j \in E_j$, $e_{j'} \in E_{j'}$, and \[ x \in X_{j,j'}:= \{  g_{a_j,b_{kj}} g_{a_{j'},b_{kj'}}^{-1} \mid \mbox{$a_j$, $a_{j'}$, and $c_k$ belong to the same face of $Y$} \}.\]
\item\label{i:at most one edge} There is at most one edge connecting each vertex $g\psi_{a_j}\!(E_j)$ to each vertex $g'\psi_{a_{j'}}\!(E_{j'})$ if and only if, for every $k$ such that $a_j$, $a_{j'}$, and $c_k$ belong to the same face of $Y$,   
\[ g_{a_j,b_{kj}}^{-1} \psi_{a_j}\!(E_j) g_{a_j,b_{kj}}   \cap g_{a_{j'},b_{kj'}}^{-1} \psi_{a_{j'}}\!(E_{j'}) g_{a_{j'},b_{kj'}} = \psi_{c_k}(F_k). \]
\end{enumerate}
\end{prop}

\noindent In the special case that $G(Y)$ is a simple complex of groups, the set $X_{j,j'}$ in \eqref{i:at least one edge} is trivial, and each edge group and face group may be identified with its image under inclusion into  the vertex group $V$.  Hence the condition in \eqref{i:at least one edge} reduces to $V = E_j E_{j'}$, and the condition in \eqref{i:at most one edge} reduces to $E_j \cap E_{j'} = F_k$.

\begin{proof}[Proof of Proposition \ref{p:complete_bipartite_product}]  We prove \eqref{i:at least one edge}; the proof of \eqref{i:at most one edge} is similarly straightforward.  Suppose that for all $g, g' \in V$ there is at least one edge of $L$ connecting the vertex $g\psi_{a_j}(E_j)$ to the vertex $g'\psi_{a_{j'}}(E_{j'})$.  
Let $h \in V$.  By Definition~\ref{d:link_local_development}, there is some $k$ such that $a_j$, $a_{j'}$, and $c_k$ belong to the same face of $Y$, and 
\[ \psi_{a_j}(E_j)g_{a_j,b_{kj}} \cap h\psi_{a_{j'}}(E_{j'})g_{a_{j'},b_{kj'}} \neq \emptyset.\]
Thus there exist $e_j \in E_j$ and $e_{j'} \in E_{j'}$ such that 
\[ h\psi_{a_{j'}}(e_{j'})g_{a_{j'},b_{kj'}} = \psi_{a_j}(e_j)g_{a_j,b_{kj}}. \]
The result follows immediately.

For the converse, let $g, g' \in V$.  Then by assumption there are elements $e_j \in E_j$ and $e_{j'} \in E_{j'}$, and some $k$ such that $a_j$, $a_{j'}$, and $c_{k}$ belong to the same face of $Y$, so that
\[g^{-1}g' = \psi_{a_j}(e_j) g_{a_j,b_{kj}} g_{a_{j'},b_{kj'}}^{-1} \psi_{a_{j'}}(e_{j'}).\]
Hence $g \psi_{a_j}(E_j)g_{a_j,b_{kj}} \cap g'\psi_{a_{j'}}(E_{j'})g_{a_{j'},b_{kj'}} \neq \emptyset$, and so by Definition \ref{d:link_local_development}, there is at least one edge between $g\psi_{a_j}(E_j)$ and  $g'\psi_{a_{j'}}(E_{j'})$.
%To prove \eqref{i:at most one edge}, assume that there is at most one edge between each vertex $g\psi_{a_j}(E_j)$ and each vertex $g'\psi_{a_{j'}}(E_{j'})$.  Then for every $k$ such that $a_j$, $a_{j'}$, and $c_k$ all belong to the same face of $Y$, we have that the intersection
%\[ g_{a_j,b_{kj}}^{-1} \psi_{a_j}(E_j) g_{a_j,b_{kj}}   \cap g_{a_{j'},b_{kj'}}^{-1} \psi_{a_{j'}}(E_{j'}) g_{a_{j'},b_{kj'}}\]
%is either empty or is a unique coset of $\psi_{c_k}(F_k)$ in $V$.  But the trivial coset is in this intersection, so the intersection is as desired.
%Conversely, suppose that for every $k$ such that $a_j$, $a_{j'}$, and $c_k$ all belong to the same face of $Y$,   
%\[ g_{a_j,b_{kj}}^{-1} \psi_{a_j}(E_j) g_{a_j,b_{kj}}   \cap g_{a_{j'},b_{kj'}}^{-1} \psi_{a_{j'}}(E_{j'}) g_{a_{j'},b_{kj'}} = \psi_{c_k}(F_k).\]
%Let $g, g' \in V$.  It is straightforward to show that if, for some $k$ such that $a_j$, $a_{j'}$, and $c_k$ all belong to the same face of $Y$, there are $h,h' \in V$ such that
%\[ h,h' \in g \psi_{a_j}(E_j) g_{a_j,b_{kj}}   \cap g' \psi_{a_{j'}}(E_{j'}) g_{a_{j'},b_{kj'}}\] 
%then $h \psi_{c_k}(F_k) = h' \psi_{c_k}(F_k)$.  Hence there is at most one edge between $g\psi_{a_j}(E_j)$ and  $g'\psi_{a_{j'}}(E_{j'})$, as required.
\end{proof}

\begin{figure}
\begin{center}
\begin{overpic}{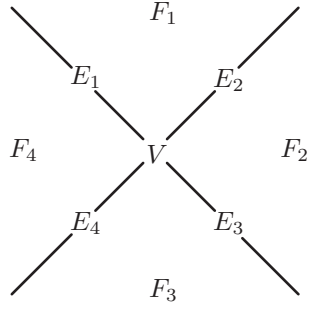}
\put(47,46){$V$}
\put(48,95){$F_1$}
\put(48,0){$F_3$}
\put(0,48){$F_4$}
\put(93,48){$F_2$}
\put(21,23){$E_4$}
\put(70,23){$E_3$}
\put(21,73){$E_1$}
\put(70,72){$E_2$}
\end{overpic}
\caption{Local groups for Corollary  \ref{c:link_K}.}
\label{fig:vertex-grouping}
\end{center}
\end{figure}

\begin{corollary}\label{c:link_K}  Suppose that $G(Y)$ is a simple complex of groups over a tessellation $Y$ of a compact orientable surface by right-angled polygons, with four such polygons meeting at each vertex.  Let $\s$ be a vertex in this tiling and let $V = G_\s$.  Let the adjacent edge groups $E_j$ and face groups $F_k$ be as in Figure \ref{fig:vertex-grouping}.  Since $G(Y)$ is simple, we may identify each $E_j$ and $F_k$ with its image under inclusion into $V$.

%Denote by $a_1$, $a_2$, $a_3$, and $a_4$, with cyclic indexing, the four edges of $Y'$ such that $t(a_j) = \s$ and $i(a_j)$ is the midpoint of an edge of $Y$.  Write $G_{i(a_j)} = E_j$ for these edge groups.  Since $G(Y)$ is simple, we may by abuse of notation identify each $E_j$ with the subgroup $\psi_{a_j}(E_j)$ of $V$.

%Denote by $b_1$, $b_2$, $b_3$, and $b_4$, with cyclic indexing, the four edges of $Y'$ such that $t(b_j) = \s$ and $i(b_j)$ is the barycenter of a face of $Y$, with $i(b_j)$ in the face of $Y$ between $a_j$ and $a_{j+1}$, for $j = 1,2,3,4$ with cyclic indexing.  Write $G_{i(b_j)} = F_j$ for these face groups.  Since $G(Y)$ is simple, we may by abuse of notation identify each $F_j$ with the subgroup $\psi_{b_j}(b_j)$ of $V$.

Then the local development $\St(\tilde\sigma)$ has link $L$ the complete bipartite graph $\K$ if and only if
\[ V = E_1E_2 = E_2E_3 = E_3E_4 = E_4E_1,\]
\[  E_1 \cap E_2 = F_1, \quad E_2 \cap E_3 = F_2, \quad E_3 \cap E_4 = F_3, \quad E_4 \cap E_1 = F_4, \]
  and
\[ |V:E_1| + |V:E_3| = \val = |V:E_2| + |V:E_4|. \]
\end{corollary}

\begin{proof}  The link $L$ is a bipartite graph, with its two vertex sets being $V/E_1 \sqcup V/E_3$ and $V/E_2 \sqcup V/E_4$.  %Now apply Proposition~\ref{p:complete_bipartite_product} above to each face in turn.
\end{proof}

\section{Indexings and necessary conditions on a complex of groups}\label{s:indexings}

This section introduces the combinatorial data of an indexing of a complex of groups, which generalizes the indexing of a graph of groups in, for example, Bass--Kulkarni \cite{BK90}.  After defining indexings, we  establish several necessary conditions on indexings, in order for the associated complex of groups to have universal cover $\I$.  These conditions are  \emph{$\val$--thickness} (Section \ref{ss:thickness}), \emph{parallel transport} (Section \ref{ss:parallel}), and \emph{unimodularity} (Section \ref{ss:unimodularity}).  The unimodularity equation, referred to in the introduction, is derived in Section \ref{ss:unimodularity_equation}, where we also discuss existence and non-existence of solutions to the unimodularity equation.

Let $Y$ be a polygonal complex with barycentric subdivision $Y'$.  An \emph{indexing} $\Ind$ of $Y'$ is an assignment of a positive integer $\Ind(a)$ to every edge $a$ of $Y'$.  Suppose now that there is a complex of groups $G(Y)$ over $Y$, such that for each edge $a$ of $Y'$, the monomorphism $\psi_a:G_{i(a)} \to G_{t(a)}$ has finite index image.  In this case, we will say that $G(Y)$ has \emph{finite indices}.  

\begin{definition}\label{d:indexing}
 Let $G(Y)$ be a complex of groups with finite indices.  The \emph{indexing induced by $G(Y)$} is the indexing $\Ind = \Ind_{G(Y)}$ of $Y'$ given by, for each edge $a$ of $Y'$,
\[  \Ind_{G(Y)}(a)=|G_{t(a)}:\psi_a(G_{i(a)})|.  \]
 In particular, if $G(Y)$ is a complex of \emph{finite} groups, then $G(Y)$ has finite indices, and for each edge $a$ of $Y'$
\[ \Ind_{G(Y)}(a) = \frac{|G_{t(a)}|}{|G_{i(a)}|}. \]
\end{definition}

\subsection{$\val$--thickness}\label{ss:thickness}

Let $Y$ be a polygonal complex.  Fix a positive integer $\val \geq 2$.

\begin{definition}
An indexing $\Ind$ of $Y'$ is \emph{$\val$--thick} if for every vertex $\s$ of $Y'$ such that $\s$ is the midpoint of an edge of $Y$,  \[ \val = \sum_{\stackrel{a \in E(Y')}{t(a) = \s}} \Ind(a). \]
\end{definition}

In the case of graphs of groups, $\val$--thickness at each vertex is a necessary and sufficient condition for the universal covering tree to be $\val$--regular  \cite[Section 1]{BK90}.  For complexes of groups, we have:

\begin{lemma}\label{l:v_thickness} Let $G(Y)$ be a developable complex of groups over $Y$ such that the universal cover of $G(Y)$ is Bourdon's building $\I$.  Then the induced indexing $\Ind = \Ind_{G(Y)}$ is $\val$--thick.
\end{lemma}

\begin{proof} First note that since $\I$ is locally finite, the complex of groups $G(Y)$ necessarily has finite indices.  Let $\s$ be a vertex of $Y'$ such that $\s$ is the midpoint of an edge of $Y$, and let $\tilde\s$ be any lift of $\s$ in (the barycentric subdivision of) $\I$.  By definition of $\I$, there are exactly $\val$ distinct faces of $\I$ which contain $\tilde\s$.  By the construction of the universal cover $ \widetilde{G(Y)} \cong\I$, the faces of $\I$ which contain $\tilde\s$ correspond bijectively to the cosets
\[ \coprod_{\stackrel{a \in E(Y')}{t(a) = \s}} G_{\s}/\psi_a(G_{i(a)}).\]
Hence $\Ind_{G(Y)}$ is $\val$--thick.
\end{proof}

\begin{corollary}\label{c:v_thickness} Let $Y$ be a tessellation of a compact orientable surface by regular right-angled hyperbolic $\p$--gons.  Let $G(Y)$ be a developable complex of groups over $Y$ with universal cover $\I$, and let $\Ind = \Ind_{G(Y)}$ be the indexing induced by $G(Y)$.  Then for every pair of distinct edges $a$ and $b$ in $Y'$ such that $t(a) = t(b)$ is the midpoint of an edge of $Y$,
\[ \Ind(a) + \Ind(b) = \val. \]
\end{corollary}

\subsection{Parallel transport}\label{ss:parallel}

Let $Y$ be a polygonal complex.  We first define an equivalence relation on the set of edges of $Y'$.

\begin{definition}\label{d:parallel}
Suppose that $a$, $b$, $a'$, and $b'$ are four distinct edges in $Y'$, such that $a$ and $b$ are composable, $b'$ and $a'$ are composable, and $ab = b'a'$.
Then we say that $a$ is \emph{parallel} to $a'$, and $b$ is \emph{parallel} to $b'$.  This relation generates an equivalence relation on the set of edges of $Y'$, which we call \emph{parallelism}.
%\end{definition}
% \begin{definition} 

An indexing $\Ind$ of $Y'$ has \emph{parallel transport} if it is constant on each parallelism class of edges in $Y'$.
\end{definition}

In our setting, the equivalence classes have the following geometric description.  Let $Y$ be a tessellation of a compact orientable surface by regular right-angled hyperbolic $\p$--gons.  Let $h$ be a closed, oriented geodesic consisting of edges in the tessellation $Y$.  The geodesic $h$ has a  collar neighborhood which is an (immersed) annulus in $S$. Thus the collar of $h$  has a well-defined \emph{right side} and \emph{left side}. There will then be one parallelism class containing all the edges of $Y'$ that come into $h$ from the right, that is, their terminal vertex is on $h$ and their initial vertex is to the right of $h$. Another, distinct parallelism class contains all the edges of $Y'$ that come into $h$ from the left.  Note that the sets of edges of $Y'$ that come into $h$ from the left and from the right are disjoint, because $S$ is orientable.  All together, the number of parallelism classes of edges in $Y'$ is twice the number of closed geodesics in $Y$.

\begin{lemma}\label{l:parallel} Let $Y$ be a tessellation of a compact orientable surface by regular right-angled hyperbolic $\p$--gons.  Let $G(Y)$ be a developable complex of groups over $Y$ such that the universal cover of $G(Y)$ is $\I$.  Then the induced indexing $\Ind = \Ind_{G(Y)}$ has parallel transport.
\end{lemma}

\noindent Thus if $h$ is a closed geodesic of $Y$, every edge of $Y'$ that comes into $h$ from the right will have the same index $n$. By Corollary \ref{c:v_thickness}, every edge of $Y'$ that comes into $h$ from the left is forced to have index $v-n$.

\begin{proof}[Proof of Lemma \ref{l:parallel}]  It suffices to prove that $\Ind(a) = \Ind(a')$ for edges $a$ and $a'$ as in Definition~\ref{d:parallel}.  Let $\s = t(a)$.  Note that since $Y$ is $2$--dimensional, $\s$ is a vertex of the tessellation $Y$.  Write $G_\s = V$, $G_{i(a)} = E_1$, $G_{t(a')} = E_2$, and $G_{i(a')} = F$.  By abuse of notation, we identify $E_1$ with the subgroup $\psi_a(E_1)$ of $V$ and identify $F$ with the subgroup $\psi_{a'}(F)$ of $E_2$.  We now wish to show that $|V:E_1| = |E_2:F|$.

Fix a lift $\tilde\s$ of $\s$ in the universal cover $\I$.  Note that $\tilde\s$ is a vertex of $\I$.  We will color a subset of the edges of $\I$ which are adjacent to $\tilde\s$ as follows.  Use red to color the edges of $\I$ adjacent to $\tilde\s$ which project to the edge of $Y$ with midpoint $i(a)$, and use blue to color the edges of $\I$ adjacent to $\tilde\s$ which project to the edge of $Y$ with midpoint $t(a')$.

Since the link of $\tilde\s$ in $\I$ is the complete bipartite graph $\K$, every blue edge is connected to every red edge by exactly one face of $\I$.  Moreover, by construction of the universal cover, and since there is exactly one face of $Y$ lying between the edges with midpoints $i(a)$ and $t(a')$, every face of $\I$ which connects a red edge to a blue edge projects to the face of $Y$ with barycenter $i(a')$.
Therefore, for any blue edge, the number of faces of $\I$ which are attached to this edge and which project to the face of $Y$ with barycenter $i(a')$ is equal to the total number of red edges.  By construction of the universal cover, the number of such faces of $\I$ is equal to the index of $F$ in $E_2$, and the total number of red edges is equal to the index of $E_1$ in $V$.  Hence $|V : E_1| = |E_2 : F|$, as required.
\end{proof}

\subsection{Unimodularity}\label{ss:unimodularity}

Let $Y$ be a polygonal complex and $\Ind$ an indexing of $Y'$.
Let $\ell$ be an oriented closed loop in the $1$--skeleton of $Y'$.  To the loop $\ell$ we associate a positive rational number $\Ind(\ell)$, which is the product of the integers $\Ind(f)$ for each edge $f$ traversed in $\ell$ from initial vertex to terminal vertex, and the rational numbers $\Ind(b)^{-1}$ for each edge $b$ traversed in $\ell$ from terminal vertex to initial vertex.  That is, $\Ind(\ell)$ is the product of the indices of the edges traversed forwards in $\ell$, divided by the product of the indices of the edges traversed backwards in $\ell$.

\begin{definition}\label{def:unimodular}
The indexing $\Ind$ of $Y'$ is \emph{unimodular} if $\Ind(\ell) = 1$ for every oriented closed loop $\ell$ in $Y'$.
\end{definition}

Definition \ref{def:unimodular} is suggested by the corresponding criterion for graphs of groups  \cite[Section 1]{BK90}. It is clear that  for a complex of finite groups $G(Y)$, the induced indexing $\Ind=\Ind_{G(Y)}$ is unimodular.

%	The following result, suggested by the unimodularity condition in Section~1 of \cite{BK90}, is clear.

%	\begin{lemma}\label{l:unimodular} Let $G(Y)$ be a complex of finite groups over $Y$.  Then the induced indexing $\Ind=\Ind_{G(Y)}$ is unimodular. \end{lemma}

%	

%	
\subsection{The unimodularity equation}\label{ss:unimodularity_equation}

Let $Y$ be a tessellation of a surface and $\Ind$ an indexing of $Y'$.  We now derive a polynomial equation, called the unimodularity equation, which must have integer solutions in order for there to be a complex of groups $G(Y)$ with universal cover $\I$ inducing this indexing.  We then discuss existence and non-existence of solutions to the unimodularity equation.

% (see Corollary~\ref{c:unimodularity} for a precise statement).  In Section \ref{ss:solutions unimodularity} we discuss solutions to the unimodularity equation.  Then in Section~\ref{ss:no_solutions_unimodularity} we establish two cases in which the unimodularity equation has no solution, and discuss why we cannot effectively determine the set of solutions to this equation.
%	
%	\subsection{Derivation of the unimodularity equation}\label{ss:derivation}

%	Let $Y$ be a tessellation of a compact orientable surface by regular right-angled hyperbolic $\p$--gons.  Let $\ell$ be an oriented closed loop in the $1$--skeleton of $Y'$, such that the vertices in $\ell$ alternate between barycenters of faces of $Y$ and midpoints of edges of $Y$.  Let $f_1,\ldots,f_k$ be the edges of $Y'$ which are traversed forward (from initial vertex to terminal vertex) in $\ell$, and let $b_1,\ldots,b_k$ be the edges of $Y'$ which are traversed backward (from terminal vertex to initial vertex) in $\ell$, so that $t(f_j) = t(b_j)$, and $i(f_{j+1}) = i(b_j)$. Here, indices $j=1, \ldots, k$ are numbered modulo $k$.  We say that the closed circuit $\ell$ has length $k$.

In the following result, the circuits $\ell$ of interest are oriented closed loops in the dual $1$--skeleton to $Y$. That is, each such $\ell$ is an oriented closed loop in the $1$--skeleton of $Y'$, such that the vertices in $\ell$ alternate between barycenters of faces of $Y$ and midpoints of edges of $Y$.   We may also think of such an $\ell$ as being an oriented closed circuit in the dual graph to the tiling $Y$. 

\begin{lemma}\label{l:unimodularity}
Let $Y$ be a tessellation of a compact orientable surface by regular right-angled hyperbolic $\p$--gons.  Let $G(Y)$ be a developable complex of groups over $Y$ with universal cover $\I$, and let $\Ind = \Ind_{G(Y)}$ be the indexing induced by $G(Y)$. For an oriented closed circuit $\ell$ of length $k$ in the dual $1$--skeleton to $Y$, let $f_1,\ldots,f_k$ be the edges of $Y'$ which are traversed forwards in $\ell$. Then the closed circuit $\ell$ induces an equation 
       \begin{equation}\label{e:unimodularity} \prod_{j=1}^{k} a_j \: = \:  \prod_{j=1}^{k} (\val-a_j), \quad\mbox{ where } a_j = \Ind(f_j) \in \mathbb{N} \mbox{ and } 1 \leq a_j < \val. \end{equation}
We call this equation the \emph{unimodularity equation}.
\end{lemma}

\begin{proof}
For $j = 1, \ldots, k$, let $f_1,\ldots,f_k$ be the edges of $Y'$ which are traversed forwards in $\ell$, and let $b_1,\ldots,b_k$ be the edges of $Y'$ which are traversed backwards in $\ell$, so that $t(f_j) = t(b_j)$. Then, by Corollary \ref{c:v_thickness} ($\val$--thickness), $\Ind(b_j) = \val - \Ind(f_j)$. Upon setting $a_j = \Ind(f_j)$, we obtain
\begin{equation}\label{e:uni-product}
\prod_{j=1}^k \frac{a_j}{\val - a_j} \: = \: 
\prod_{j=1}^k \frac{\Ind(f_j)}{\Ind(b_j)} \: = \:
\Ind(\ell) \: = \: 1.
\end{equation}

\vspace{-3ex}
\end{proof}

%	Assume that $\Ind$ is both unimodular and $\val$--thick, for some fixed $\val \geq 2$.  Then by the definitions of unimodularity and $\val$--thickness, \[ \prod_{j=1}^k \Ind(a_j) = \prod_{j=1}^k \Ind(b_j) = \prod_{j=1}^k (\val - \Ind(a_j)). \]
%	By abuse of notation, we write $a_j$ for $\Ind(a_j)$.  Then every circuit of length $k$ in the dual graph to $Y$ induces an equation
%	       \begin{equation}\label{e:unimodularity} \prod_{j=1}^{k} a_j = \prod_{j=1}^{k} (\val-a_j) \quad\mbox{ where } a_j \in \mathbb{N} \mbox{ and } 1 \leq a_j < \val. \end{equation}
%	We call this equation the \emph{unimodularity equation}.  Our derivation of this equation, together with Lemma~\ref{l:unimodular} and Corollary~\ref{c:v_thickness} above, yields:

%	\begin{corollary}\label{c:unimodularity} Let $Y$ be a tessellation of a compact orientable surface by regular right-angled hyperbolic $\p$--gons.  Let $G(Y)$ be a developable complex of groups over $Y$ with universal cover $\I$, and let $\Ind = \Ind_{G(Y)}$ be the indexing induced by $G(Y)$.  Then for every circuit of length $k$ in the dual graph to $Y$, there is a solution $1 \leq a_j = \Ind(a_j) < \val$ to the unimodularity equation~\eqref{e:unimodularity}.
%	\end{corollary}

% \subsection{Existence of solutions to the unimodularity equation}\label{ss:solutions unimodularity}

\begin{remark}

\begin{enumerate}\item In the following cases, there are ``easy" solutions to the unimodularity equation:
\begin{enumerate}
\item\label{i:k even} When $k$ is even, put $a_2 = \val - a_1$, $a_4 = \val - a_3$, and so forth.
\item When $\val$ is even, put each $a_j = \val/2$.
\end{enumerate}
\item\label{i:multiply_b} There are some solutions when both $\val$ and $k$ are odd. For example, for any even integer $b$, let $\val = (b + 1)(b^2 + 1)$.  Then the unimodularity equation when $k = 3$ has a solution, not necessarily unique, given by $a_1 = b^2(b + 1)$ and $a_2 = a_3 = (b^2 + 1)$.
\item\label{i:multiply_v} For all $k$, if there is a solution for $\val$, then there is a solution for the same $k$ and any positive integer multiple of $\val$.
\item\label{i:add_2k} For all $\val$, if there is a solution for $k$, then there is a solution for the same $\val$ and $k' = k+2$. Simply set $a_{k+2} = \val - a_{k+1}$.
\end{enumerate}
\end{remark}

%	There are also solutions for some values of $\val$ and $k$ both odd.  An example is $\val = 15$ and $k = 3$, where the unique solution (up to obvious symmetries) is $a_1 = 12$ and $a_2 = a_3 = 5$.  More generally, for any even integer $b$, let $\val = (b + 1)(b^2 + 1)$.  Then the unimodularity equation when $k = 3$ has a solution (not necessarily unique) given by $a_1 = b^2(b + 1)$ and $a_2 = a_3 = (b^2 + 1)$.

% \subsection{Non-existence of solutions to the unimodularity equation}\label{ss:no_solutions_unimodularity}

We now consider cases in which solutions to the unimodularity equation do not exist.  The main results below are Lemma \ref{l:unimodularity_power_of_prime} and Corollary \ref{c:inf many v not power prime}, which each provide an infinite family of odd integers $\val \geq 3$ such that when $k$ is odd, the unimodularity equation has no solution.  After proving Corollary \ref{c:inf many v not power prime}, we briefly discuss why we cannot effectively determine the set of solutions to the unimodularity equation.  The proofs of Corollaries \ref{c:no_solutions_some_p} and \ref{c:inf many v not power prime} and this discussion draw on private communications with Michael Broshi and Roger Heath-Brown.

\begin{lemma}\label{l:unimodularity_power_of_prime} Let $\val$ be a power of an odd prime.  Then the unimodularity equation~\eqref{e:unimodularity}  has no solutions when  $k$ is odd. \end{lemma}

\begin{proof}  Let $\val = q^n$ where $q$ is an odd prime, and suppose that there is a solution $a_1,\ldots,a_k$  to~\eqref{e:unimodularity}.  Let $v_j$ be the $q$--valuation of $a_j$. That is, $a_j = q^{v_j} c_j$, where $c_j$ is relatively prime to $q$ and $v_j < n$.  Upon dividing the unimodularity equation through by $q^{v_1 + \ldots + v_k}$, we obtain
\[        \prod_{j=1}^{k} c_j = \prod_{j=1}^{k} (q^{n - v_j} - c_j).\]
Now reduce modulo $q$ to get \[        \prod_{j=1}^{k} c_j \equiv \prod_{j=1}^{k}  (-c_j) \equiv -\prod_{j=1}^{k} c_j   \mod q,\]
since $k$ is odd.  Since $q$ is odd, this means
\[       \prod_{i=1}^{k} c_j \equiv 0  \mod q. \]
But by construction, none of the $c_j$ is divisible by $q$. Contradiction. \end{proof}

To obtain additional infinite families of odd integers $\val \geq 3$ for which the unimodularity equation has no solution when $k$ is odd, we begin with the following statement.

\begin{lemma}\label{l:p_necessary_condition} Suppose that $\val = 3q^n$, where $q$ is an odd prime, and that $k$ is odd. If the unimodularity equation~\eqref{e:unimodularity} has a solution, then there is an integer $m$ such that
\[ 2^m \equiv (-1)^{m+1} \mod q.\]
\end{lemma}

\begin{proof}
If $q=3$, then by Lemma~\ref{l:unimodularity_power_of_prime}, there is no solution to equation \eqref{e:unimodularity}. Thus we may assume $q > 3$.

Let $v_j$ be the $q$--valuation of $a_j$, as in the proof of Lemma~\ref{l:unimodularity_power_of_prime}, with $a_j = q^{v_j}c_j$, and $c_j$ relatively prime to $q$. Then, after dividing the unimodularity equation~\eqref{e:unimodularity}  by $q^{v_1 + \ldots + v_k}$, we obtain
\begin{equation}\label{e:reduced-unimodularity}
\prod_{j=1}^k c_j \: = \:  \prod_{j=1}^k \left(\frac{\val}{q^{v_j}} - c_j\right).
\end{equation}

We may reorder the coefficients $a_j$ such that $v_j = n$ for $1 \leq j \leq l$, and $v_j < n$ for $l < j \leq k$. For all indices $j \leq l$, we have
$$a_j = q^{n}c_j < \val = 3 q^n.$$
Thus, for the first $\l$ indices, $c_j = 1$ or $c_j=2$. After further reordering, we may assume that $c_j = 1$ for $1 \leq j \leq l'$ and $c_j = 2$ for $l'< j \leq l$. After these simplifications, equation \eqref{e:reduced-unimodularity} takes the form
$$ \prod_{j=1}^{l'} 1 \, \prod_{j=l'+1}^l 2 \,  \prod_{j=l+1}^k c_j 
\: = \:  \prod_{j=1}^{l'} 2 \, \prod_{j=l'+1}^l 1 \, \prod_{j=l+1}^k \left(\frac{\val}{q^{v_j}} - c_j\right). 
$$
Now, recall that for all $i > l$, the term $\val / q^{v_j}$ is still divisible by $q$. Thus reducing modulo $q$ gives 
\begin{equation}\label{e:reduced-by-q}
2^{l - l'} \,  \prod_{j=l+1}^k c_j 
\: \equiv \:   2^{l'}  \, \prod_{j=l+1}^k (- c_j ) \: \mod q. 
\end{equation}

For the rest of the proof, we work in the finite field $\F_q$. Since $2$ and every $c_j$ is relatively prime to $q$, each one has a multiplicative inverse in $\F_q$. Thus equation \eqref{e:reduced-by-q} simplifies to
$$2^{l - 2l'} \: \equiv \:     (-1)^{k-l}  \: \mod q. $$
Finally, because $k$ is odd, we have
$$(-1)^{k-l} = (-1)^{1-l} = (-1)^{1+l} = (-1)^{1+l - 2l'}.$$
Thus setting $m = l - 2l'$ completes the proof.
\end{proof}

%	\begin{corollary}\label{c:p'_equals_3}  Under the assumptions of Lemma~\ref{l:p_necessary_condition} above, suppose in addition that $\val = 3q^n$, equivalently that $q' = 3$.  If the unimodularity equation~\eqref{e:unimodularity} has a solution, then for some integer $0 \leq m < k$,  \[ 2^m \equiv (-1)^{m+1} \mod q.\]
%	\end{corollary}

%	\begin{proof} Note first that by Lemma~\ref{l:unimodularity_power_of_prime} above,
%	 since $q' = 3$ is prime,~\eqref{e:unimodularity_p'} above has no solution for $k$ odd.  Hence by Lemma~\ref{l:p_necessary_condition} above,
%	\[  (-1)^{k - l}\prod_{j=1}^l (3 - c_j)  \equiv  \prod_{j = 1}^l c_j    \mod q \]
%	where $1 \leq l < k$ and each $c_j$ is either $1$ or $2$.
%	Suppose that exactly $l'$ of the $c_js$ are equal to $2$, where $0 \leq l' \leq l$.  Then
%	\[ (-1)^{k - l}2^{l - l'} \equiv 2^{l'}  \mod q.\]  Since $q > 2$, we may  without loss of generality assume that $l' \leq l - l'$, divide through by $(-1)^{k-l}2^{l'}$, and obtain
%	\[ 2^{l - 2l'} \equiv (-1)^{k - l}  \mod q. \]
%	As $k$ is odd, this implies $2^m \equiv (-1)^{m+1}  \pmod q$ for $m = l - 2l'$.  Recalling that $1 \leq l < k$ and $0 \leq l' \leq \frac{l}{2}$, we obtain $0 \leq m < k$, as desired.
%	\end{proof}

\begin{corollary}\label{c:no_solutions_some_p}  Suppose that $\val = 3q^n$ where $q$ is a prime such that the multiplicative order of $2$ in $(\mathbb{Z}/q\mathbb{Z})^\times$ is congruent to $2 \pmod 4$.  If $k$ is odd, then there is no solution to the unimodularity equation~\eqref{e:unimodularity}.
\end{corollary}

\begin{proof}  Let the multiplicative order of $2$ in $(\mathbb{Z}/q\mathbb{Z})^\times$ be $4j + 2$, where $j$ is a positive integer.  Then $2^m \equiv -1 \pmod q$ if and only if for some positive integer $l$, $m = (4j + 2)l + 2j + 1$, which is odd. Thus, for $m$ odd, there are no solutions to $2^m \equiv (-1)^{m+1} \pmod q$. Similarly,  $2^m \equiv 1 \pmod q$ if and only if for some positive integer $l$, $m = (4j + 2)l$, which is even. Thus, for $m$ even, there are also no solutions to $2^m \equiv (-1)^{m+1} \pmod q$.  The conclusion then follows from Lemma \ref{l:p_necessary_condition}.
\end{proof}

\begin{corollary}\label{c:inf many v not power prime}  Let $q$ be any of the infinitely many prime numbers of the form $q \equiv 3 \pmod 8$. Then the unimodularity equation has no solution for $v=3q^n$ and $k$ odd. 
%	There are infinitely many odd integers $\val \geq 3$ which are not powers of primes such that the unimodularity equation \eqref{e:unimodularity} has no solution when $k$ is odd.
\end{corollary}

\begin{proof}  
%	If $q=3$, then the result is true by Lemma \ref{l:unimodularity_power_of_prime}. Thus we may assume $q>3$.
%
By Dirichlet's Theorem, there are infinitely many primes $q \equiv 3 \pmod 8$.
If $q \equiv 3 \pmod 8$ is prime, then $(\mathbb{Z}/q\mathbb{Z})^\times$ is a cyclic group of order $2 \pmod 8$.
Furthermore, for any odd prime $q$, the equation $x^2\equiv 2 \pmod q$ only has solutions when $q \equiv \pm 1 \pmod 8$. In particular, when $q \equiv 3 \pmod 8$, the number $2$ represents an odd power of a generator of $(\mathbb{Z}/q\mathbb{Z})^\times$, hence the order of $2$ is $2 \pmod 4$. Thus, by Corollary \ref{c:no_solutions_some_p}, the unimodularity equation has no solution for $v=3q^n$ and $k$ odd.
\end{proof}

Given the result of Corollary \ref{c:no_solutions_some_p}, one might consider, for simplicity, just the primes $q \equiv 3 \pmod 4$ such that $2$ is a generator of $(\mathbb{Z}/q\mathbb{Z})^\times$, hence $2$ has order $q - 1 \equiv 2 \pmod 4$.  However, it is unknown whether there are infinitely many such primes $q$.  Indeed, a celebrated conjecture of E.\ Artin (with history dating back to Gauss) holds that $2$ generates $(\mathbb{Z}/q\mathbb{Z})^\times$ for infinitely many primes $q$.  This conjecture was proved by Hooley under the assumption of the Generalized Riemann Hypothesis for Dedekind zeta functions \cite{hooley}.
In particular, Hooley's argument guarantees the existence of infinitely many primes of the form $q \equiv 3 \pmod 4$, such that $2$ is a generator of $(\mathbb{Z}/q\mathbb{Z})^\times$.  We remark that there are also infinitely many primes $q \not \equiv 3 \pmod 4$ with the property that the multiplicative order of $2$ in $(\mathbb{Z}/q\mathbb{Z})^\times$ is  congruent to $2 \pmod 4$, but there is no effective description of these primes. Thus even in the very special case $\val = 3q$ with $q > 3$ prime, we cannot effectively determine the set of solutions to the unimodularity equation when $k$ is odd.

\section{Homology and tessellations}\label{s:homology}

In this section, we use the results of Sections \ref{s:group theory} and \ref{s:indexings}  to prove Theorem \ref{prop:homology-summary} and Corollary \ref{cor:no-lattice-v3} of the introduction.  Proposition \ref{prop:weak-homology} below establishes \eqref{i:weak intro} of Theorem \ref{prop:homology-summary}, and Proposition \ref{prop:strong-homology} below establishes \eqref{i:strong intro} of Theorem \ref{prop:homology-summary}.  Corollary \ref{cor:no-lattice-v3} then follows from Proposition \ref{c:F not divisible by 4} below.

Fix $S = S_g$ a compact, orientable surface of genus $g \geq 2$.  Let $Y$ be a tiling of $S_g$ by $\f$ copies of a regular right-angled hyperbolic $\pgon$.  Let $h_1, \ldots, h_n$ be the closed geodesics of this tiling, each one with prescribed orientation. The geodesic $h_i$ represents a homology class $[h_i] \in H_1(S, \Z) \cong \Z^{2g}$.
Occasionally, it will also help to view the vector $[h_i] \in \Z^{2g} \subset \R^{2g}$ as an element of $H_1(S, \R)$.

\begin{prop}\label{prop:weak-homology}
The following conditions are equivalent:
\begin{enumerate}
\item\label{i:exist c_i} There exist non-zero coefficients $c_1, \ldots, c_n \in \Z \setminus \{0\}$ such that $\sum c_i [h_i] = 0$.
\item\label{i:some odd v} There exists an odd integer $\val \geq 3$ and a simple, faithful complex of finite groups $G(Y)$ over $Y$ such that the link of every local development is $\K$.
\item\label{i:lattice for some v} There exists an odd integer $\val \geq 3$ and a uniform lattice $\G$ in $\Aut(\I)$, such that $Y \cong \G \backslash \I$. Furthermore, $\pi_1(Y)$ is a subgroup of $\G$.
\item\label{i:indexing for some v} There exists an odd integer $\val \geq 3$ and an indexing $\Ind$ of the barycentric subdivision $Y'$ of $Y$, such that $\Ind$ is $\val$--thick, unimodular, and has parallel transport.
\end{enumerate}
\end{prop}

Note that the implication \eqref{i:indexing for some v} $\Rightarrow$ \eqref{i:some odd v} can be taken as a satisfying converse to the results in Section \ref{s:indexings}. However, one must be careful with quantifiers: if an indexing $\Ind$ is unimodular, has parallel transport, and is $\val$--thick for some $\val$, then there is a complex of finite groups $G(Y)$ such that the link of every local development is $K_{w,w}$, for some odd integer $w$ that may differ from $\val$.

\begin{proof}
\eqref{i:exist c_i} $\Rightarrow$ \eqref{i:some odd v}: Suppose there exist non-zero coefficients $c_1, \ldots, c_n \in \Z \setminus \{0\}$ such that $\sum c_i [h_i] = 0$. After reversing the orientation on some of the geodesics $h_i$, we may assume that all the coefficients $c_i$ are positive. Rather than proving that \emph{there exists} an odd integer $\val$ and a complex of finite groups over $Y$ with local developments  having link $\K$, we will prove the same conclusion \emph{for all} values of $\val$ that take a particular form.

Fix an integer $b \geq 1$, and let $\val$ be any positive integer that is divisible by $(b^{c_i} + 1)$ for every coefficient $c_i$. One way to produce an odd integer $\val$ of this form is to choose an even $b$, and then set $\val = \prod_{i=1}^{n} (b^{c_i} + 1)$. Once $\val$ is fixed, for every $i$ we define
$$k_i =\frac{\val}{b^{c_i} + 1}.$$ 
Our criterion for $\val$ is designed specifically to ensure that each $k_i$ will be an integer.

%	
%	\begin{equation}\label{eq:v-product}
%	\val := \prod_{i=1}^{n} (b^{c_i} + 1), \qquad k_j := \prod_{i\neq j} (b^{c_i} + 1) = \frac{\val}{b^{c_j} + 1}  \: .
%	\end{equation}
%	Note that with this definition, $\val$ is odd.

Since $\sum c_i [h_i]$ is trivial in $ H_1(S, \Z)$, it must be the boundary of an integral $2$--chain $\Delta$. That is, $\Delta = \sum_{j=1}^{\f} d_j F_j$ is a sum of faces with integer coefficients. We assume that all faces are positively oriented, i.e.\ that $\sum_{j=1}^{\f} F_j$ is the fundamental class in $H_2(S, \Z)$. After adding a multiple of the fundamental class to $\Delta$, we may assume that all coefficients $d_j$ are non-negative, and that at least one coefficient $d_j$ is zero.

We are now ready to assign local groups. Give each face $F_j$ the local group $(\Z/b)^{d_j}$, the direct product of $d_j$ copies of the cyclic group $\Z/b$. Now, let $e$ be an edge of the tiling, oriented the same way as its ambient geodesic $h_i$. Because we have assumed that $\bdy \left( \sum d_j F_j \right) = \sum c_i h_i$, we must have $c_i = d_j -  d_k $,
where $F_j$ is the face to the left of $e$ and $F_k$ is the face to the right of $e$. In words, the face to the left of $e$ must have a higher coefficient than the face to the right of $e$, with a difference of $c_i$.

\begin{figure}%[ht]
\begin{center}
\begin{overpic}{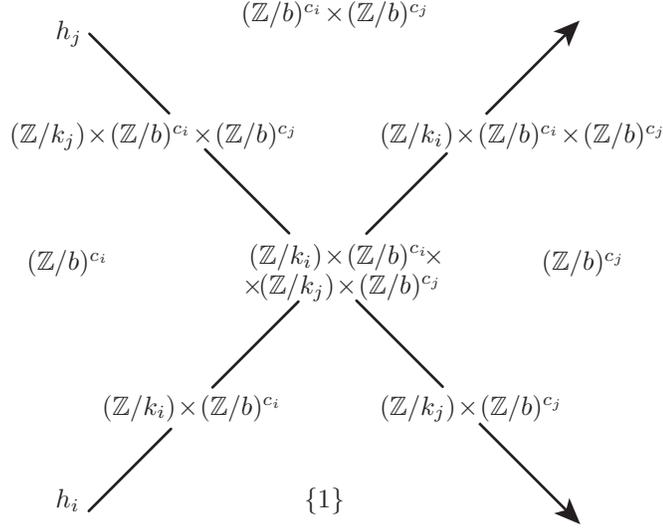}
\put(41,87){$(\Z/b)^{c_i} \cross (\Z/b)^{c_j}$}
\put(52,3){$\{ 1 \}$}
\put(4,44){$ (\Z/b)^{c_i}$}
\put(93,44){$(\Z/b)^{c_j}$}
\put(42,40){$\stackrel{\displaystyle{(\Z/k_i) \cross (\Z/b)^{c_i} \cross }}{ \cross (\Z/k_j) \cross (\Z/b)^{c_j}}$}
\put(9,3){$h_i$}
\put(9,84){$h_j$}
\put(1,66){$(\Z/k_j) \cross (\Z/b)^{c_i} \cross (\Z/b)^{c_j}$}
\put(65,66){$(\Z/k_i) \cross (\Z/b)^{c_i} \cross (\Z/b)^{c_j}$}
\put(17,19){$(\Z/k_i) \cross (\Z/b)^{c_i}$}
\put(65,19){$(\Z/k_j) \cross (\Z/b)^{c_j}$}
\end{overpic}
\caption{Local groups in the neighborhood of a vertex $\sigma$. Every local group near $\sigma$ also has a direct product factor of $(\Z/b)^{d(\sigma)}$, which is not shown in the figure.}
\label{fig:local-groups}
\end{center}
\end{figure}

Let $\sigma$ be a vertex of the tiling, and let $h_i$ and $h_j$ be the two geodesics that intersect at $\sigma$. (These may turn out to be the same geodesic, but we keep the indices $i$ and $j$ distinct for notational purposes.) Suppose, without loss of generality, that $h_j$ crosses $h_i$ from left to right, as in Figure \ref{fig:local-groups}. Let $d(\sigma)$  be the smallest coefficient on any of the four faces adjacent to $\sigma$. Then, if the four faces are positioned as in the figure,  their coefficients (going clockwise from East) are $d(\sigma)+c_j$, $d(\sigma)$, $d(\sigma)+c_i$, and $d(\sigma)+c_i+c_j$. The factor $(\Z/b)^{d(\sigma)}$ is present throughout, and is suppressed in Figure \ref{fig:local-groups}. We assign the local groups to edges and the vertex as shown in the figure.  All the  monomorphisms are the ``obvious'' inclusions of the corresponding cyclic groups.  Hence this complex of groups is simple.

Let us check that $G(Y)$ satisfies the criteria of Corollary \ref{c:link_K} above. In every oriented edge group along geodesic $h_i$, the index of the face group on the left is $k_i$ and the index of the face group on the right is $k_i b^{c_i}$. Thus the sum of these indices is
$$k_i + k_i b^{c_i} = \val,$$
by the definition of $k_i$. Similarly, for every oriented edge along geodesic $h_j$, the sum of indices from the face groups on the left and right is $k_j + k_j b^{c_j} = \val,$ as desired. Additionally, every face group is the intersection of the adjacent edge groups, and the vertex group at $\sigma$ is the group-theoretic product of every pair of consecutive edge groups. Thus, by Corollary \ref{c:link_K} above, the link of every local development is $\K$.

This complex of groups is faithful, because there is a face $F_j$ with coefficient $d_j=0$, hence with trivial face group. Therefore \eqref{i:some odd v} holds.

\medskip

\eqref{i:some odd v} $\Rightarrow$ \eqref{i:lattice for some v}: This is immediate from Proposition \ref{p:surface_subgroups} and Corollary \ref{c:lattice}.

\medskip

\eqref{i:lattice for some v} $\Rightarrow$ \eqref{i:indexing for some v}: By Corollary \ref{c:lattice}, there is a developable complex of finite groups $G(Y)$ over $Y$, whose universal cover is $\I$. Then the indexing $\Ind_{G(Y)}$ induced by $G(Y)$ is $\val$--thick by Corollary \ref{c:v_thickness} and has parallel transport by Lemma \ref{l:parallel}. Furthermore, the indexing induced by a complex of finite groups is unimodular.

\medskip

\eqref{i:indexing for some v} $\Rightarrow$ \eqref{i:exist c_i}: Suppose that, for some odd integer $\val \geq 3$, there is an indexing $\Ind$ of the barycentric subdivision $Y'$ of $Y$, such that $\Ind$ is $\val$--thick, has parallel transport, and 
is unimodular. We will first find real-valued coefficients $r_1, \ldots, r_n$ such that $\sum r_i [h_i] = 0$, and then find integer coefficients with the same property.

For each oriented geodesic $h_i$, recall from Section \ref{ss:parallel} that there is one parallelism class of edges of $Y'$ that come into $h_i$ from the left, and another parallelism class of edges of $Y'$ that come into $h_i$ from the right. By parallel transport, every edge that comes in from the left has the same index $a_i$; by $\val$--thickness, every edge that comes in from the right has the same index $\val - a_i$. Furthermore, because $\val$ is odd, $a_i \neq \val - a_i$. Given this setup, we assign each geodesic $h_i$ the real-valued coefficient 
$$r_i = \log (a_i) - \log (\val-a_i) \neq 0.$$

To see that $\sum r_i [h_i]$ vanishes in homology, we recall the definition of the intersection pairing on $H_1(S, \R)$. Given a pair of oriented closed curves $\alpha, \beta$ whose intersections in $S$ are transverse, the \emph{intersection pairing} 
$\langle \alpha, \,  \beta  \rangle$
is defined to be the number of times that $\beta$ crosses $\alpha$ from left to right, minus the number of times that $\beta$ crosses $\alpha$ from right to left. It is clear that this definition is skew-symmetric, and depends only on the homology classes of $\alpha$ and $\beta$. Furthermore, the pairing $\langle \alpha, \,  \beta  \rangle$ extends linearly to real-valued combinations of closed curves, and gives a non-degenerate, skew-symmetric, bilinear pairing on $H_1(S, \R) \cong \R^{2g}$.

Now, let $\alpha = \ell$ be an oriented closed circuit in the dual graph to $Y$, and let $\beta = \sum r_i [h_i]$, with coefficients $r_i$ as above. Then the intersection pairing $\langle \ell, \, \sum r_i [h_i]  \rangle$ is  exactly the logarithm of the unimodularity product $\Ind(\ell)$ computed  in equation \eqref{e:uni-product} of Lemma \ref{l:unimodularity}. Thus, because the indexing $\Ind$ is unimodular, we have 
$$\left\langle \ell, \: \sum r_i [h_i]  \right\rangle \: = \: \log \, \Ind(\ell) \: = \: 0,$$
for every closed circuit $\ell$ in the dual graph to $Y$. Therefore, since every integer homology class is represented by a circuit in the dual graph to $Y$, and in particular, such circuits span $H_1(S)$, the non-degeneracy of the pairing means that $\sum r_i [h_i] = 0  \in H_1(S, \R)$.

For each geodesic $h_i$, the homology class $[h_i]$ is an element of $H_1(S, \Z) \cong \Z^{2g} \subset \R^{2g}$. Thus we may view $[h_i]$ as an integer vector in $\R^{2g}$. Let $M$ be the $2g \times n$ matrix whose columns are $[h_1], \ldots, [h_n]$. We already know that the column vectors of $M$ are linearly dependent (over $\R$), hence the null space of $M$ is non-empty. But since every entry of $M$ is an integer, the null space $N(M)$ is spanned by rational vectors, and rational vectors are dense in $N(M)$. Thus there exist rational coefficients $q_1, \ldots, q_n$, with each $q_i$ arbitrarily close to $r_i$ (and in particular $q_i \neq 0$ for all $i$), such that $\sum q_i [h_i] = 0$. By clearing the denominators, we obtain non-zero integer coefficients $c_1, \ldots, c_n$ with the same property.
\end{proof}

The stronger statement of \eqref{i:exist c_i} $\Rightarrow$ \eqref{i:some odd v} that was mentioned near the beginning of the proof is as follows:

\begin{corollary}\label{cor:stronger}
 Let $b \geq 1$ be any positive integer, and suppose there are integer coefficients $c_i \geq 1$, such that $\sum c_i [h_i] = 0$. Then there is a simple, faithful complex of finite groups $G(Y)$ over $Y$ such that the link of every local development is $\K$, for every value of $\val$ that is divisible by $(b^{c_i}+1)$ for each $i$. Consequently, for each such $\val$ divisible by all $(b^{c_i}+1)$, there is a uniform lattice $\G$ in $\Aut(\I)$, such that $Y \cong \G \backslash \I$. 
\end{corollary}

\begin{prop}\label{prop:strong-homology}
The following conditions are equivalent:
\begin{enumerate}
\item\label{i:choice sign} There exist choices of sign $c_1, \ldots, c_n \in  \{ \pm 1\}$ such that $\sum c_i [h_i] = 0$.
\item\label{i:lattice for odd v} For every integer $\val \geq 2$, there is a uniform lattice $\G$ in $\Aut(\I)$, such that $Y \cong \G \backslash \I$. Furthermore, $\pi_1(Y)$ is a subgroup of $\G$.
\item\label{i:all odd v} For every integer $\val \geq 2$, there is a simple, faithful complex of finite groups $G(Y)$ over $Y$ such that the link of every local development is $\K$.
\item\label{i:v equals 3} For $\val = 3$, there is a complex of finite groups $G(Y)$ over $Y$ such that the link of every local development is $K_{3,3}$.
\end{enumerate}
\end{prop}

\begin{proof}
We will prove \eqref{i:choice sign} $\Rightarrow$ \eqref{i:all odd v} $\Rightarrow$ \eqref{i:lattice for odd v} $\Rightarrow$ \eqref{i:v equals 3} $\Rightarrow$ \eqref{i:choice sign}.

\medskip

\eqref{i:choice sign} $\Rightarrow$ \eqref{i:all odd v}: Suppose there exist coefficients $c_1, \ldots, c_n \in  \{ \pm 1\}$ such that $\sum c_i [h_i] = 0$. After reversing the orientation on some of the geodesics $h_i$, we may assume that $c_i = 1$ for all $i$. Choose any integer $\val \geq 2$, and let $b = \val -1$. Then, for every $i$, $\val$ is divisible by
$$b^{c_i} + 1 = b+1 = v.$$
Thus, by Corollary \ref{cor:stronger} above, there is a (simple, faithful) complex of finite groups $G(Y)$ over $Y$ such that the link of every local development is $\K$.

\medskip

\eqref{i:all odd v} $\Rightarrow$ \eqref{i:lattice for odd v}: Immediate from Proposition \ref{p:surface_subgroups} and Corollary \ref{c:lattice}.

\medskip

\eqref{i:lattice for odd v} $\Rightarrow$ \eqref{i:v equals 3}: Immediate by restricting \eqref{i:lattice for odd v} to $v=3$, and applying Corollary \ref{c:lattice}.

\medskip

\eqref{i:v equals 3} $\Rightarrow$ \eqref{i:choice sign}: Suppose that there is a complex of finite groups $G(Y)$ over $Y$, such that the link of every local development is $K_{3,3}$. Let $G_j$ be the local group on face $F_j$, and consider the $2$--chain $\Delta = \sum_{j=1}^\f (\log \abs{G_j}) \, F_j$, where $\abs{G_j}$ is the number of elements in $G_j$. Let $C$ be the real-valued $1$--chain $C = \bdy \Delta$. Then, for every oriented edge $e$ of the tiling, the coefficient of $e$ in $C$ is $\log \abs{G_j} - \log \abs{G_k}$, where $G_j$ is the group on the left and $G_k$ is the group on the right of $e$.
By Lemma \ref{l:parallel}  (parallel transport), every edge in a closed geodesic $h_i$ will have the same coefficient. Thus we may write $\bdy \Delta = \sum r_i h_i$, hence $\sum r_i [h_i] = 0 \in H_1(S, \R)$.

Let $G_e$ be the local group on edge $e$. Then, by Corollary \ref{c:v_thickness}  ($\val$--thickness),
$$[G_e: G_j] + [G_e: G_k] = 3,$$ which means one of the groups $G_j$ or $G_k$ is twice as large as the other. Thus
$$r_i = \log \abs{G_j} - \log \abs{G_k} = \pm \log 2.$$
By setting $c_i = r_i / \log 2$, we obtain coefficients $c_1, \ldots, c_n \in  \{ \pm 1\}$, such that $\sum c_i [h_i] = 0$.
\end{proof}

\begin{prop}\label{c:F not divisible by 4}  Suppose that $\f=\frac{8(g-1)}{\p - 4}$ is not divisible by $4$. Then
\begin{enumerate}
\item\label{i:no Gamma p 3 g} a lattice $\G_{\p,3,g}$ does not exist; and
\item\label{i:no Y all odd v} there is no tessellation $Y$ of $S_g$ by $\f$ copies of a regular right-angled hyperbolic $\pgon$ such that for every odd integer $\val \geq 3$, there is a lattice $\Gg$ with $\Gg \bs \I \cong Y$.
\end{enumerate}
\end{prop}

\begin{proof}  We prove the contrapositive.  If a lattice $\G_{\p,3,g}$ does exist, then the action of $\G_{\p,3,g}$ on $I_{\p,3}$ induces a complex of finite groups over the tiling $Y = \G_{\p,3,g} \bs I_{\p,3}$ of $S_g$, such that the link of every local development is $K_{3,3}$.  Similarly, if there exists a tessellation $Y$ of $S_g$ by $\f$ copies of a regular right-angled $\pgon$ such that for every odd integer $\val \geq 3$, there is a lattice $\Gg$ with $\Gg \bs \I \cong Y$, then for every odd integer $\val \geq 3$ there is a complex of finite groups $G(Y)$ over $Y$ such that the link of every local development is $\K$.

In both cases, by Proposition \ref{prop:strong-homology}  there exist choices of sign $c_i \in \pm 1$ such that $\sum c_i[h_i] = 0$, where $h_i$ are the geodesics of the tiling $Y$.  That is, there exists a $2$--chain $\Delta$ of faces with integer weights, such that the boundary $\partial\Delta$ is a sum of geodesics $h_i$ with coefficients $\pm 1$.

The weights on adjacent faces have to differ by $1$. Thus, for any vertex $\s$ of $Y$, there is some integer $n$ such that in the neighborhood of $\s$, there are two (opposite) faces of weight $n$, one face of weight $(n-1)$, and one face of weight $(n+1)$.  In this case, we will say that $\s$ is a \emph{vertex of type $(n)$}. Let $a_n$ be the total number of vertices of type $(n)$.

\begin{lemma}\label{l:a_n divisible by p} For every integer $n$, the number $a_n$ of vertices of type $(n)$ is divisible by $\p$.\end{lemma}

\begin{proof}[Proof of Lemma \ref{l:a_n divisible by p}] Let $f_n$ be the number of faces with weight $n$. Every vertex of every such face is a vertex of type $(n-1)$, or $(n)$, or $(n+1)$. Every vertex of type $(n)$ meets two corners of faces with weight $n$, whereas those of type $(n-1)$ or $(n+1)$ meet one corner of a face with weight $n$. Thus the total number of corners of faces with weight $n$ is
\begin{equation}\label{eq:face-corners}
a_{n-1} + 2 a_n + a_{n+1} = \p f_n.
\end{equation}

We may now prove the lemma by induction on $n$.  Since the surface $S_g$ is compact, there is a smallest $n$ such that $a_n$ is nonzero. Without loss of generality, suppose that the first nonzero $a_n$ is $a_1$. This is equivalent to the smallest weight on any face being $0$.

For the base case of the induction, consider equation \eqref{eq:face-corners} with $n=0$. By hypothesis, $a_{-1} = a_0 = 0$. Meanwhile, $\p f_0$ is clearly divisible by $\p$. Thus, by equation \eqref{eq:face-corners}, $a_1$ is divisible by $p$.

For the inductive step, assume that $a_{n-1}$ and $a_n$ are divisible by $\p$. Since the right-hand  side of \eqref{eq:face-corners} is divisible by $\p$, so is $a_{n+1}$. Thus, by induction, Lemma \ref{l:a_n divisible by p} is proved.
%
%	Let $f_0$ be the number of faces of weight $0$. Since $0$ is the smallest possible weight, every vertex of every face with weight $0$ is a vertex of type $(1)$. Conversely, a vertex is of type $(1)$ if and only if it meets a face of weight $0$. Thus the number of vertices of type $(1)$ is equal to
%	\[ a_1 = \p f_0, \]
%	which is divisible by $\p$. This proves the base case of the induction.
%
%	Now, suppose that $a_n$ is divisible by $\p$ for every $n \leq k$. Let $f_k$ be the number of faces of weight $k$. Every vertex of every such face is a vertex of type $(k-1)$, or $(k)$, or $(k+1)$. The total number of corners of faces with weight $k$ is
%	\[        a_{k-1} + 2 a_k + a_{k+1} = \p f_k. \]
%	By induction, $a_{k-1}$ and $a_k$ are divisible by $\p$. Since the right-hand side is divisible by $\p$, so is $a_{k+1}$. This completes the proof of Lemma \ref{l:a_n divisible by p}. 
\end{proof}

An immediate consequence of Lemma \ref{l:a_n divisible by p} is that the total number of vertices in the tessellation $Y$ is divisible by $\p$. But since every vertex is $4$--valent and every face has $\p$ sides, it follows that the number of faces in $Y$ is divisible by $4$.  That is, $\f \equiv 0 \pmod 4$.
\end{proof}

\section{Proof of the Main Theorem}\label{s:proof}

In this section, we apply results from Sections \ref{s:group theory}--\ref{s:homology} above to prove the Main Theorem, along with  Corollary \ref{cor:gromov-special-case}.  We begin the argument in Section \ref{ss:construct tiling}, by constructing particular tessellations of a surface $S$, which will satisfy the homological conditions of Propositions \ref{prop:weak-homology} or \ref{prop:strong-homology}. In Section \ref{ss:existence proof}, we use these tessellations to prove the existence results \eqref{i:existence} of the Main Theorem, along with  Corollary \ref{cor:gromov-special-case}. In Section \ref{ss:nonexistence proof}, we prove the non-existence results \eqref{i:non-existence} of the Main Theorem.

%	Corollary \ref{cor:gromov-special-case} of the introduction will follow 
%	from Corollaries \ref{c:v_even} and \ref{c:f_even} below.  Let $\p \geq 5$, 
%	$\val \geq 2$, and $g \geq 2$ be integers, and assume that $\f = \frac{8(g-1)}{p - 4}$ 
%	is a positive integer.  Let $S_g$ be a compact orientable genus $g$ surface.

\subsection{Constructing tessellations}\label{ss:construct tiling}
Let $\p \geq 5$, $\val \geq 2$, and $g \geq 2$ be integers, and assume that $\f = \frac{8(g-1)}{p - 4}$ is a positive integer.  Let $S_g$ be a compact orientable surface of genus $g$. Then a theorem of Edmonds--Ewing--Kulkarni states that there exists some tiling $Y$ of $S_g$ by $\f$ regular right-angled hyperbolic $\p$--gons \cite{EEK}. However, our goal is to construct a particular tiling, such that the closed geodesics of $Y$ will satisfy our homology conditions from Section \ref{s:homology}. 

For every $\p \geq 5$, let $\HH^2_p$ be a rescaled copy of the hyperbolic plane $\HH^2$, in which a regular right-angled $\p$--gon has sidelength exactly $1$. In the following construction, it will be convenient to work with $\p$--gons modeled on $\HH^2_p$ rather than on $\HH^2$. This way, after several polygons are glued together to form a surface with geodesic boundary, we can measure the length of the boundary in terms of units of sidelength.

\begin{prop}\label{p: tiling f divisible by 4}
Let $\p \geq 5$, and suppose that  $\f = \frac{8(g-1)}{p-4}$ is divisible by $4$. Then the surface $S_g$ admits a tiling $Y$ by $\f$ copies of a regular right-angled $\p$--gon, such that the geodesics $h_1, \ldots, h_n$ of this tiling satisfy $\sum [h_i] = 0 \in H_1(S_g)$.
% $Y$ satisfies the condition of Proposition \ref{prop:strong-homology}.
\end{prop}

\begin{proof}
We will consider two cases: $\p$ odd and $\p$ even. In each case, we will construct a closed surface out of $\f$ copies of a regular right-angled $\p$--gon. So long as all the gluings reverse orientation (i.e., so long as outward normals to a polygon are glued to inward normals), the resulting surface will be orientable. By Euler characteristic considerations, this surface will be of genus $g$, where $(g-1) = \f(p-4)/8$.

\medskip

\noindent \emph{\underline{Case 1: $\p$ odd.} }
If $\p$ is odd and $\f = \frac{8(g-1)}{p-4}$ is an integer, then $\f$ must actually be divisible by $8$. We begin by coloring the edges of each $\pgon$ as follows: one edge is colored red, while the remaining $\p-1$ edges are colored alternately yellow and black. Since $\p$ is odd, each polygon has $(p-1)/2$ black edges, and the same number of yellow edges. The gluing will respect these colors.

To begin the construction, we glue the $\pgons$ in pairs, along their (unique) red edges. We now have $\f/2$ copies of a right-angled $q$--gon, where $q=2\p-4$. The sides of each $q$--gon are colored  alternately yellow and black, with one yellow edge of length $2$ and another, opposite black edge of length $2$ (in the metric of $\HH^2_p$). 

The next step of the construction is to glue the $q$--gons in pairs along all of their black sides, by ``super-imposing" one $q$--gon above another.  The result is $\f/4$ copies of a $(q/2)$--holed sphere, where all boundary circles are yellow. One boundary circle will have length $4$, while the remaining $(q/2 - 1) = (\p-3)$ boundary circles have length $2$. (See Figure \ref{fig:sphere-matching}.) Arrange these $\f/4$ spheres cyclically, and number them $1$ through $\f/4$. Notice that $\f/4$ is even. For each odd-numbered sphere, glue the long boundary circle (of length 4) to the corresponding boundary circle of the next sphere. Glue the $\p-3$ short boundary circles (of length $2$) to the corresponding circles of the previous sphere. This gluing gives a  tessellation $Y$ of the surface $S_g$.

\begin{figure}%[ht]
\begin{center}
\begin{overpic}{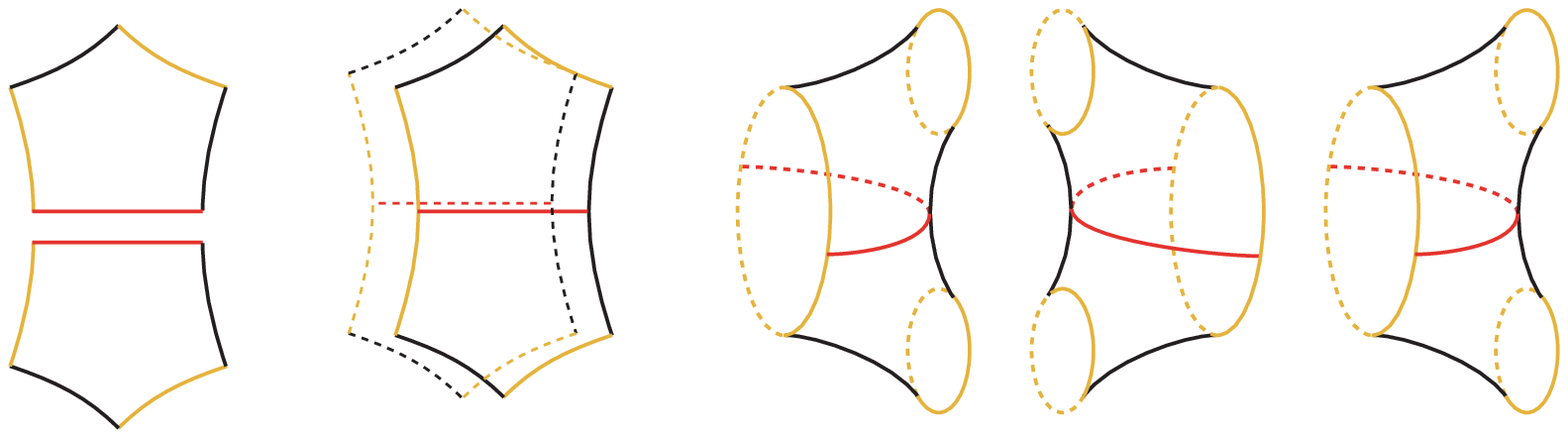}
\put(12,25){(A)}
\put(37,25){(B)}
\put(75,25){(C)}
\end{overpic}
\caption{Three steps in the construction of a tiling, for odd $\p$. (A) Match $\pgons$ along red edges to form $(2\p-4)$--gons. (B) Double along black edges to form $(p-2)$ holed spheres. (C) Glue cuffs of spheres along yellow edges to form closed surface. In this example, $p=5$.}
\label{fig:sphere-matching}
\end{center}
\end{figure}

The edges of $Y$ can be partitioned into embedded closed geodesics, as follows. The red edges form $\f / 8$ closed geodesics, each of which has length 4. When we cut $Y$ along these geodesics, the result is two connected components: these are the top and bottom halves in Figure \ref{fig:sphere-matching}(C). Thus the red geodesics can be oriented so that their sum is $0$ in $H_1(S_g)$. Meanwhile, the black edges also form several disjoint geodesics that cut $Y$ into two connected components: these are the front and back halves in Figure \ref{fig:sphere-matching}(C). Thus the black geodesics can also be oriented so that their sum is $0$ in $H_1(S_g)$. 

Finally, the yellow edges form several disjoint closed geodesics that cut $Y$ into the $\f /4$ copies of a $(q/2)$--holed sphere that we had before the last step of the construction. If we orient each yellow geodesic as the oriented boundary of an odd-numbered holed sphere, then the sum of all the yellow geodesics bounds the sum of the odd-numbered holed spheres. Thus the sum of all the geodesics is homologically trivial, as desired.

\medskip

\noindent \emph{\underline{Case 2: $\p$ even.} } The construction in this case is nearly the same as for $\p$ odd, except that we skip the very first gluing along red edges. First, we color the edges of each $\pgon$ alternately yellow and black. Second, we glue the $\p$--gons in pairs along all of their black sides, by ``super-imposing" one $\p$--gon above another. The result is $\f/2$ copies of a $(\p/2)$--holed sphere, where all boundary circles are yellow and have length 2. Arrange these $\f/2$ spheres cyclically, and number them $1$ through $\f/2$. Notice that by hypothesis, $\f/2$ is even.  For each odd-numbered sphere, glue one boundary circle to the next sphere and the remaining $(\p/2 -1)$ boundary circles to the previous sphere. This gluing gives a  tessellation $Y$ of the surface $S_g$.

As before, the edges of $Y$ can be partitioned into embedded closed geodesics.  The black edges form several disjoint geodesics that cut $Y$ into two connected components: these are the front and back halves in the ring of holed spheres. Thus the black geodesics can also be oriented so that their sum is $0$ in $H_1(S_g)$. Meanwhile, if we orient each yellow geodesic as the oriented boundary of an odd-numbered holed sphere, then the sum of all the yellow geodesics is also $0$ in $H_1(S_g)$.
\end{proof}

\begin{prop}\label{p: tiling f composite}
Let $\p \geq 5$, and suppose that  $\f = \frac{8(g-1)}{p-4}$ is a composite integer. Then the surface $S_g$ admits a tiling $Y$ by $\f$ copies of a regular right-angled $\p$--gon, such that the geodesics $h_1, \ldots, h_n$ of this tiling satisfy $\sum c_i [h_i] = 0 \in H_1(S_g)$, with coefficients $c_i \in \{1, 2\}$.
% such that $Y$ satisfies the condition of Proposition \ref{prop:weak-homology}.
\end{prop}

Note that when $\f$ is not divisible by $4$, Proposition \ref{c:F not divisible by 4} implies that there does not exist a tiling $Y$ whose geodesics   satisfy $\sum [h_i] = 0$. Thus finding weights of the form $c_i \in \{1, 2\}$ can be seen as a ``best possible'' outcome when $\f$ is not divisible by $4$.

\begin{proof}
If $\f$ is divisible by $4$, then the result is true by Proposition \ref{p: tiling f divisible by 4}. Thus we may assume that $\f \equiv 2 \pmod 4$, or $\f$ is odd. In either case, since $\f(p-4) = 8(g-1)$ is divisible by 8, $\p$ must be divisible by $4$. In particular, since $p \geq 5$ and is divisible by $4$, we actually know that $p \geq 8$.

As in Proposition \ref{p: tiling f divisible by 4}, we begin by coloring the edges of each $\pgon$. Initially, we color the edges alternately yellow and red. Then, we choose two opposite red edges and paint them black. For the duration of the construction, we will draw each $\pgon$ as a concave jigsaw puzzle piece, in the shape of a long rectangle with $(\p-4)/4$ semi-circular scoops taken out of the top side and  $(\p-4)/4$ semi-circular scoops taken out of the bottom side. The two left and right vertical edges are colored black, the horizontal edges are all colored yellow, and the semi-circular scoops are all colored red. See Figure \ref{fig:torus-matching}.

\begin{figure}%[ht]
\begin{center}
\begin{overpic}{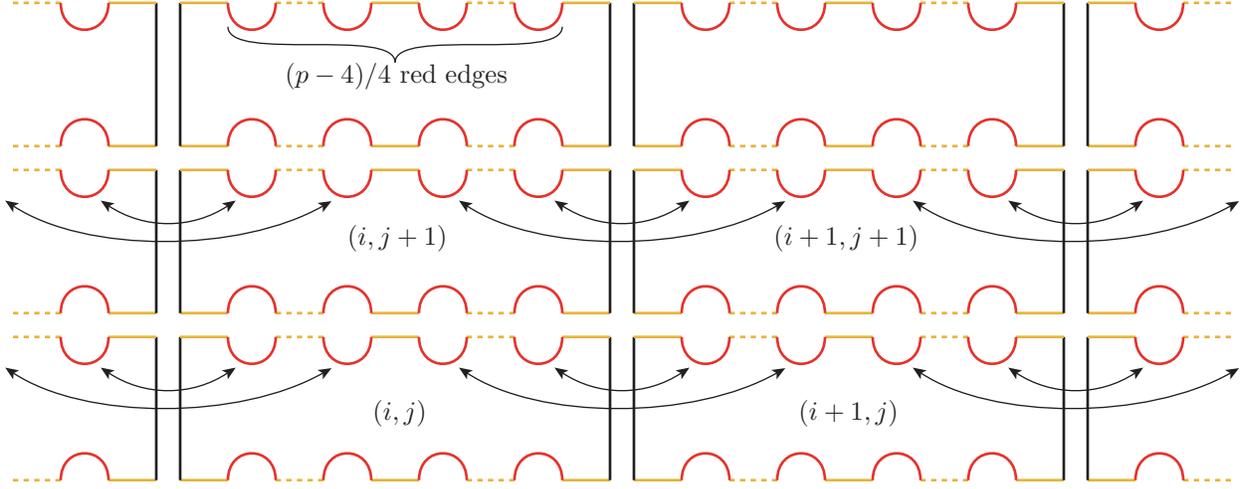}
\put(23,32){$(\p-4)/4$ red edges}
\put(30,5){$(i,j)$}
\put(28,19){$(i,j+1)$}
\put(64,5){$(i+1,j)$}
\put(62,19){$(i+1,j+1)$}
\end{overpic}
\caption{$\f = xy$  right-angled $\pgons$ are arranged into $x$ columns and $y$ rows to form a torus with $F(\p-4)/4$ holes. Then, the holes are glued in pairs.}
\label{fig:torus-matching}
\end{center}
\end{figure}

Choose positive integers $x$ and $y$ such that $xy=\f$. For the moment, to construct a tessellation, we do not yet require $\f$ to be composite: thus one or both of $x,y$ may equal $1$. Since $\f$ is not divisible by 4, we may adopt the convention that $y$ is always odd. Hence $\f$ is even if and only if $x$ is even.

We construct a tiling as follows. First, arrange the long jigsaw puzzle pieces into $x$ columns and $y$ rows. Number the columns cyclically $1, 2, \ldots, x$ and the rows cyclically $1, 2, \ldots, y$. Then we glue the adjacent black edges: the right side of the polygon at position $(i,j)$ is glued to the left side of the polygon at position $(i+1, j)$. Similarly, we glue the adjacent yellow edges: each horizontal edge on top of  the polygon at position $(i,j)$ is glued to the corresponding horizontal edge on the bottom of the polygon at position $(i, j+1)$. Here, the horizontal indices are taken modulo $x$, and the vertical indices are taken modulo $y$.  The result of this gluing is a torus with $\f(\p-4)/4$ holes, where the boundary of each hole is a red geodesic of length $2$.

To match up the red geodesics, we consider two cases: $\p \equiv 4 \pmod 8$ and $\p \equiv 0 \pmod 8$.

If $\p \equiv 4 \pmod 8$, then $(\p-4)/4$ must be even. Hence the top of each jigsaw puzzle piece has an even number of red edges, as does the bottom of each piece. Thus we may glue the right-most $(\p-4)/8$ red edges on top of the polygon  at position $(i,j)$ to the left-most $(\p-4)/8$ red edges on top of the polygon at position $(i+1,j)$. We perform the identical gluing for the red edges on the bottom of each jigsaw piece. In each row, this gluing creates   $(\p-4)/8$ ``handles'' between the $j^\mathrm{th}$ column and the $(j+1)^\mathrm{st}$ column, with the cross-section of each handle being a red geodesic of length $2$. In particular, we obtain a closed orientable surface by adding $\f(\p-4)/8$ handles to a torus, hence the genus is $g = 1 + \f(\p-4)/8$, as desired.

If $\p \equiv 0 \pmod 8$, then $(\p-4)/4$ must be odd. Thus, since $g-1 = \f(p-4)/8 $ is an integer, $\f$ must be even. By convention, this means that $x$ (the number of columns) is even. For each jigsaw puzzle piece at position $(i,j)$, where $i$ is even, we glue the right-most $\p/8$ red edges on top of that polygon to the left-most $\p/8$ red edges on top of the polygon at position $(i+1,j)$. We glue the left-most $(\p-8)/8$ red edges on top of that polygon to the right-most $(\p-8)/8$ red edges on top of the polygon at position $(i-1,j)$. We perform the identical gluing for the red edges on the bottom of each jigsaw piece. As above, the effect of this gluing is to add $\f(\p-4)/8$ handles to a torus, with each handle connecting consecutive columns.

Observe that this construction of a tiling did not require $\f$ to be composite. However, this hypothesis will be used to assign weights to geodesics, in order to satisfy the desired homology condition. For the rest of the proof, we \emph{do} assume that $\f$ is composite, and that $x,y \geq 2$. 

Consider the $2$--chain $\Delta_y$ obtained as a sum of faces with the following weights. Every jigsaw puzzle piece 
at position $(i,1)$ receives a weight of $2$. For indices $j > 1$, every jigsaw piece 
at position $(i,j)$ receives a weight of $(j \mod 2)$.  Thus all the tiles in the same row have the same weight. Since $y$ is odd, the sequence of weights on different rows is $[2,0,1,0, \ldots, 1]$; thus the tiles in adjacent rows have weights that differ by $1$ or $2$. The boundary $\bdy \Delta_y$ is a weighted sum of edges that separate adjacent rows -- that is, a weighted sum of the yellow geodesics.   (This is where we use the hypothesis that $y > 1$; otherwise, there would be only one row, and the weight on the yellow geodesics would be $0$.) Because the weights on adjacent rows differ by $1$ or $2$, all the yellow geodesics (oriented appropriately) have a weight of $1$ or $2$. Thus the sum of yellow geodesics, with these weights, is $0$ in $H_1(S_g)$.

In a similar fashion, consider the $2$--chain $\Delta_x$ obtained as follows. Every jigsaw puzzle piece 
at position $(1,j)$ receives a weight of $2$. For indices $i > 1$, every jigsaw piece 
at position $(i,j)$ receives a weight of $(i \mod 2)$.  Thus all the tiles in the same column have the same weight. Because $x > 1$, the tiles in adjacent columns have weights that differ by $1$ or $2$. The boundary $\bdy \Delta_x$ is a weighted sum of edges that separate adjacent columns -- that is, a weighted sum of the black and red geodesics. Because the weights on adjacent columns differ by $1$ or $2$, all the black and red geodesics (oriented appropriately) have a weight of $1$ or $2$. We have now assigned a weight of $1$ or $2$ to all the geodesics of the tiling, such that their sum, with these weights, is trivial in $H_1(S_g)$.
\end{proof}

As a byproduct of the above construction, we recover the following result, due to Edmonds--Ewing--Kulkarni \cite{EEK}.

\begin{corollary}\label{c:tiling-exists}
Let $\f \geq 1$, $\p \geq 5$, and $g \geq 2$ be integers. Then the closed orientable surface $S_g$ of genus $g$ admits a tiling by $\f$ regular right--angled $\pgons$ if and only if
$$\f (\p - 4) = 8 (g-1).$$
\end{corollary}

\begin{proof}
$(\Rightarrow)$ Suppose that $S_g$ admits a tiling by $\f$ regular right--angled $\pgons$. The number of edges in this tiling is $E = \f \p / 2$, and the number of vertices is $V = \f \p /4$, because every vertex is $4$--valent. Thus
$$-2(g-1) \: = \: \chi(S_g) \: = \: \frac{\f \p}{4} - \frac{ \f \p} {2} + \f,$$
which simplifies to the desired equation.

$(\Leftarrow)$ For the converse, suppose that $\f (\p - 4) = 8 (g-1)$. If $\f$ is divisible by $4$, then Proposition \ref{p: tiling f divisible by 4} constructs a tiling of $S_g$ by $\f$ right-angled $\pgons$. If $\f$ is not divisible by $4$, then Proposition \ref{p: tiling f composite} constructs a tiling of $S_g$ by $\f$ right-angled $\pgons$. Recall that the construction of the tiling in Proposition \ref{p: tiling f composite} did not use the hypothesis that $\f$ is composite.
\end{proof}

\subsection{Existence of lattices}\label{ss:existence proof}
The following result gives parts \eqref{i:existence v even}, \eqref{i:existence f divisible by 4}, and \eqref{i:existence f composite} of the Main Theorem.

\begin{theorem}\label{thm:existence}
Let $\p \geq 5$, $\val \geq 2$, and $g \geq 2$ be integers, and let $\I$ be Bourdon's building. Let $\f = \frac{8(g-1)}{p - 4}$.
\begin{enumerate}
\item[(a)] If $\val \geq 2$ is even, then for all integers $\f$, a lattice $\Gg$ exists.
\item[(b)] If $\f$ is divisible by $4$, then for all integers $\val \geq 2$, a lattice $\Gg$ exists.
\item[(c)] If $\f$ is composite, and $\val$ is divisible by $(b+1)(b^2 +1)$, where $b$ is a positive even number, then  a lattice $\Gg$ exists.
\end{enumerate}
\end{theorem}

\begin{proof}
For part (a), let $\p \geq 5$, and let $\f$ be any positive integer. Let $Y$ be any tessellation of an orientable surface $S$ by $\f$ regular right-angled hyperbolic $\p$--gons. This tessellation has to exist by Corollary \ref{c:tiling-exists}.

Let $\val = 2m$ be an even number.  We will construct a complex of groups $G(Y)$ as follows.  Each face group is trivial, hence $G(Y)$ is faithful.  Each edge group is $C_m$, the cyclic group of order $m$, and each vertex group is the direct product of two copies of $C_m$.  Let $\s$ be a vertex of $Y$ and let $h_i$ and $h_j$ be the two geodesics of the tiling which intersect at $\s$.  (These may turn out to be the same geodesic, but we keep the indices $i$ and $j$ distinct for notational purposes.)
The edge groups $C_m$ on an edge of $h_i$ include to the first $C_m$--factor in the vertex group $G_\s = C_m \times C_m$, and the edge groups $C_m$ on an edge of $h_j$ include to the second $C_m$--factor in $G_\s$.  Since all monomorphisms are inclusions, the complex of groups $G(Y)$ is simple.  By Corollary~\ref{c:link_K}, the local development $\St(\tilde\s)$ has link $\K$. Thus, by Corollary \ref{c:lattice}, the fundamental group of $G(Y)$ is a uniform lattice $\G = \Gg$ such that  $Y \cong \G \backslash \I$.

For part (b), suppose that $\f$ is divisible by $4$. Then by Proposition \ref{p: tiling f divisible by 4}, there is a tessellation $Y$ whose geodesics $h_1, \ldots, h_n$ satisfy $\sum [h_i] = 0 \in H_1(S_g)$. Thus by Proposition \ref{prop:strong-homology}, for every $\val \geq 2$ there is a uniform lattice $\G = \Gg$ such that  $Y \cong \G \backslash \I$.

For part (c), suppose that $\f$ is composite. Then, by Proposition \ref{p: tiling f composite}, there is a tessellation $Y$ whose geodesics $h_1, \ldots, h_n$ satisfy $\sum c_i [h_i] = 0 \in H_1(S_g)$ for coefficients $c_i \in \{1, 2\}$. Thus for every positive even integer $b$ and every $\val$ divisible by $(b+1)(b^2 +1)$, Corollary \ref{cor:stronger} implies that there is  a uniform lattice $\G = \Gg$ such that  $Y \cong \G \backslash \I$.
\end{proof}

\begin{corollary}\label{c:surface-subgroup}
Let $\p \geq 5$, $\val \geq 2$,  and let $\I$ be Bourdon's building.  Then for $g=p-3$, there is a 
uniform lattice $\Gg < \Aut( \I)$ such that $\pi_1(S_g) < \Gg$.
\end{corollary}

\begin{proof}
Let $\f = \frac{8(g-1)}{p - 4} = 8.$ Since $\f$ is divisible by $4$, Proposition \ref{p: tiling f divisible by 4} implies that there there is a tessellation $Y$ by $\f$ tiles, whose geodesics $h_1, \ldots, h_n$ satisfy $\sum [h_i] = 0 \in H_1(S_g)$. 
Thus by Proposition \ref{prop:strong-homology}, there is a uniform lattice $\G = \Gg < \Aut( \I)$, such that $Y \cong \G \backslash \I$, and such that $\pi_1(S_g) < \G$.
\end{proof}

Recall the theorem of Haglund   that for all $\p \geq 6$, all uniform lattices in $\Aut(\I)$ are commensurable up to conjugacy  \cite[Theorem 1.1]{H06}.  Haglund's theorem, combined with Corollary \ref{c:surface-subgroup}, immediately implies Corollary \ref{cor:gromov-special-case} of the introduction.

%	\begin{remark} Homology, the unimodularity equation and intersection pairings.
%	\end{remark}

\subsection{Non-existence of lattices}\label{ss:nonexistence proof}

The following result immediately implies part \eqref{i:nonexistence f odd} of the Main Theorem.  
\begin{theorem}\label{thm:non-existence}
Let $\p \geq 5$, $\val \geq 2$, and $g \geq 2$ be integers, and let $\I$ be Bourdon's building. Assume that $\f = \frac{8(g-1)}{p - 4}$ is a positive integer.
\begin{enumerate}
\item[(a)] If $\f$ is odd and $\val = q^n$, where $q$ is an odd prime, then a lattice $\Gg$ does not exist.
\item[(b)] If $\f$ is odd and $\val = 3 q^n$, where $q \equiv 3 \pmod 8$ is an odd prime, then a lattice $\Gg$ does not exist.
\end{enumerate}
\end{theorem}

\begin{proof}
Suppose, for a contradiction, that a lattice $\G = \Gg$ does exist. Then $Y \cong \G \backslash \I$ is a tiling of a hyperbolic surface by $\f$ regular right-angled $\p$--gons, where $\f$ is odd. By Corollary \ref{c:lattice}, there is a complex of finite groups $G(Y)$ such that the universal cover of $G(Y)$ is $\I$.

We claim that the dual graph to the tiling $Y$ contains a circuit of odd length. For, suppose not: suppose that all circuits in the dual graph to $Y$ are of even length.  The dual graph is then bipartite, with the two sets of vertices in this bipartition colored (say) black and white.  If there are $b$ black vertices and $w$ white vertices, then since every vertex of the dual graph has valence $\p$, and every edge in the dual graph connects a black vertex to a white vertex, the number of edges in the dual graph is $b\p = w\p$.  Hence $b = w$.  But $b + w = \f$ and by hypothesis $\f$ is odd, a contradiction.  Thus the dual graph must contain a circuit of length $k$, for some odd number $k$.

By Lemma \ref{l:unimodularity}, this circuit of odd length induces a unimodularity equation  \eqref{e:unimodularity}.
% Consider the unimodularity equation  \eqref{e:unimodularity} for the circuit in the dual graph that has length $k$. 
But if $\val = q^n$, where $q$ is an odd prime, then Lemma \ref{l:unimodularity_power_of_prime} implies there is no solution to the unimodularity equation. Similarly, if $\val = 3 q^n$, where $q \equiv 3 \pmod 8$ is an odd prime, then Corollary \ref{c:inf many v not power prime} implies there is no solution to the unimodularity equation. In either case, the non-existence of a solution contradicts the existence of $\Gg$.
\end{proof}

\section{Relationships with previous examples and surface subgroups}\label{s:relationships}

We conclude with a discussion of how the lattices $\Gg$ that we have constructed relate to previous examples, in Section \ref{ss:prev examples}, and to surface subgroups, in Section \ref{ss:surface subgroups}.   

Our discussion uses covering theory for complexes of groups.  Since this theory is highly technical and is not needed elsewhere, we refer the reader to \cite[Chapters III.$\mathcal{C}$ and III. $\mathcal{G}$]{BH99} and  \cite{LT07} for the definitions and properties on which we now rely. 
Fix $\p \geq 5$ and $\val \geq 2$, and let $\G_1$ and $\G_2$ be any subgroups of $\Aut(\I)$ (not necessarily lattices).  Recall that by the discussion in Section \ref{ss:complexes_of_groups} above,  for $i = 1,2$ the group $\G_i$ is the fundamental group of a faithful, developable complex of groups $G(Z_i)$, where $G(Z_i)$ has universal cover $\I$, and $Z_i$ is the polygonal complex $Z_i = \G_i \bs \I$.   

\begin{facts}\label{facts:covering theory}  We will need the following facts from covering theory.
\begin{enumerate}
\item The group $\G_1$ is a subgroup of $\G_2$ if and only if there is a covering of complexes of groups \[ \Phi: G(Z_1) \to G(Z_2). \]
\item A covering $\Phi:G(Z_1) \to G(Z_2)$ is defined over a cellular map $f:Z_1 \to Z_2$. 
\item Given any cellular map $f:Z_1 \to Z_2$, and the first barycentric subdivision $Z_1'$ of $Z_1$, a necessary condition for the existence of a covering $\Phi:G(Z_1) \to G(Z_2)$ over $f$ is that for each vertex $\sigma$ in $Z_1'$, there is a monomorphism of local groups $G_{\sigma} \to G_{f(\sigma)}$.  In particular, if all face groups of $G(Z_2)$ are trivial, then there can only be a covering $G(Z_1) \to G(Z_2)$ if all face groups of $G(Z_1)$ are also trivial.  
\item The data for a covering $\Phi:G(Z_1) \to G(Z_2)$ over $f:Z_1 \to Z_2$ also includes a collection of group elements $\phi(a) \in G_{t(f(a))}$, for each edge $a$ of $Z_1'$, which satisfies the criteria given in \cite[Chapter III.$\mathcal{C}$, Definitions 2.4 and 5.1; see also Lemma 5.2]{BH99}.
\item If $\G_1 < \G_2$, then the index of $\G_1$ in $\G_2$ is equal to the number of sheets of any covering $\Phi:G(Z_1) \to G(Z_2)$.   If a covering $\Phi:G(Z_1) \to G(Z_2)$ is over a cellular map $f:Z_1 \to Z_2$, then the number of sheets of this covering is equal to the cardinality of $f^{-1}(\tau)$, where $\tau$ is the barycenter of any face of $Z_2$ which has trivial local group in $G(Z_2)$.  \end{enumerate}
\end{facts}

\subsection{Relationships with previous examples}\label{ss:prev examples}

We now discuss some relationships between $\Gg$ and earlier examples of lattices in $\Aut(\I)$.  For this, we first recall an example of a uniform lattice $\G_0 < \Aut(\I)$ from Section 2.1 of Bourdon \cite{B97}.  (It is explained in \cite{B97} that this example was also known to others.) 
 
\begin{our_example}  Let $G(Y_0)$ be the simple complex of groups defined as follows.  The underlying complex $Y_0$ is a regular right-angled hyperbolic $\p$--gon, with no gluings.  The face group is trivial, each edge group is $C_\val$, the cyclic group of order $\val$, and each vertex group is the direct product of the adjacent edge groups.  All monomorphisms of local groups are the natural inclusions.  Denote by $\G_0$ the fundamental group of $G(Y_0)$.  The group $\G_0$ is a graph product of $\p$ copies of $C_\val$, and is a uniform lattice in $\Aut(\I)$ acting with quotient the $p$--gon $Y_0$.  
\end{our_example}
 
Recall that in \cite[Theorem 1.1]{H06}, Haglund proved that for $\p \geq 6$ all uniform lattices in $\Aut(\I)$ are commensurable, up to conjugacy.  In fact, Haglund established this result by showing that for $\p \geq 6$, every uniform lattice in $\Aut(\I)$ is commensurable to $\G_0$, up to conjugacy.  Given this proof and the above straightforward description of $\G_0$, focusing on the relationship between $\Gg$ and $\G_0$ will shed light on the relationship between our lattices $\Gg$ and previous examples.

Now by our construction, each lattice $\Gg$ is the fundamental group of a faithful complex of finite groups $G(Y)$, where $Y$ is a tessellation of a compact genus $g$ surface $S_g$ by regular right-angled hyperbolic $\pgons$.   If we label the edges of the $\pgon$ $Y_0$ cyclically by $1,2,\ldots,\p$, then given any cellular map $f:Y \to Y_0$, each edge of the tessellation $Y$ has a label, or \emph{type}, induced by pulling back these labels using $f$.  We will say that a tessellation $Y$ of $S_g$ is \emph{type-consistent} if the edges of every $\pgon$ in $Y$ can be labelled cyclically by $1,2,...,\p$, so that the labellings of edges of adjacent $\pgons$ are compatible.  Using Facts \ref{facts:covering theory}(1) and (2) above, it follows that $\Gg = \pi_1(G(Y))$ can be a subgroup of $\G_0$ only if the tessellation $Y$ is type-consistent.  

Since the tessellations $Y$ used in our constructions of $\Gg$ in Section \ref{ss:existence proof} above were not required to be type-consistent, not all of the resulting lattices $\Gg$ embed in $\G_0$.  
However, when $\val$ is even, we have the following explicit constructions of a common finite-index subgroup of  $\Gg$ and $\G_0$. These constructions work even for $\p = 5$, when Haglund's commensurability theorem does not apply.  The case  $\val$ odd is discussed after the proof.

\begin{prop}\label{p:rel with Gamma_0} Let $\p \geq 5$ and assume that $\val \geq 2$ is even.  Let $g \geq 2$ be such that $\f= \frac{8(g-1)}{p-4}$ is a positive integer.
\begin{enumerate}
\item\label{i:g} If $\f \equiv 0 \pmod 4$, there exists a lattice $\Gg$ which is an index $\f$ subgroup of $\G_0$.
\item\label{i:g'} If $\f \equiv 2 \pmod 4$, there exists a lattice $\Gg$ which has an index $2$ subgroup $\G_{\p,\val,g'}$, where $g' = 2g -1$, so that $\G_{\p,\val,g'}$ is an index $2\f$ subgroup of $\G_0$.
\item\label{i:g''} For all $\f$, there exists a lattice $\Gg$ which has an index $4$ subgroup $\G_{\p,\val,g''}$, where $g'' = 4g - 1$, so that $\G_{\p,\val,g''}$ is an index $4\f$ subgroup of $\G_0$.
\end{enumerate}
\end{prop}

\begin{proof}  For \eqref{i:g}, let $Y$ be the tessellation constructed in Proposition \ref{p: tiling f divisible by 4}, and let $Y_0$ be a single $\pgon$. Then a cellular map $f: Y \to Y_0$ can be described as follows. If $\p$ is odd, then $f$ is the quotient map by the group of color-preserving symmetries of the tiling depicted in Figure \ref{fig:sphere-matching}. The color-preserving symmetry group acts transitively on the $\pgons$. At the same time, a single odd-sided $\pgon$ in Figure \ref{fig:sphere-matching}(A) has no color-preserving symmetries. Thus the quotient of $Y$ is indeed a single $\pgon$ $Y_0$.

Similarly, if $\p$ is even, $f$ is the quotient map by a certain subgroup of color-preserving symmetries in the tiling of Proposition \ref{p: tiling f divisible by 4}. The symmetries of interest are generated by the front-to-back reflection and the left-to-right reflection visible in Figure  \ref{fig:sphere-matching}(C). The quotient is once again a single $\pgon$ $Y_0$.

Note that for both $\p$ odd and $\p$ even, a cyclic labeling $1, \ldots, \p$ of the edges of $Y_0$ pulls back to a type-consistent labeling in $Y$. In fact, every geodesic in the tessellation of $Y$ will have a well-defined type.

Now let $G(Y)$ be the complex of groups over $Y$ constructed as in the proof of Theorem \ref{thm:existence}(a).  That is, each face group is trivial, each edge group is $C_m$ where $\val = 2m$, and each vertex group is $C_m \times C_m$.  The fundamental group of $G(Y)$ is a lattice $\Gg$.  

To embed $\Gg$ in $\G_0$, we construct a covering of complexes of groups $\Phi:G(Y) \to G(Y_0)$.  The cellular map $f:Y \to Y_0$ is described above.  The monomorphisms of local groups $G_\sigma \to G_{f(\sigma)}$ are the identity map of trivial groups if $\s$ is the barycenter of a face of $Y$, the ``doubling'' 
monomorphism $C_m \to C_{2m} = C_v$ if $\s$ is the barycenter of an edge of $Y$, and the monomorphism $C_m \times C_m \to C_{v} \times C_{v}$ induced by the monomorphisms of edge group factors if $\s$ is a vertex of $Y$.  A collection of group elements $\phi(a) \in G_{t(f(a))}$ which satisfies the criteria referred to in Fact \ref{facts:covering theory}(3) above may be constructed as follows.  Choose an orientation for each geodesic $h$ in $Y$.  For each such $h$, let $i_h$ be the type of $h$, and fix a generator $x_{h}$ for the copy of $C_v$ on the edge of type $i_h$ in the $p$--gon $Y_0$.  Now choose two edges $a_h$ and $b_h$ in $Y'$, with $a_h$ coming into $h$ from the left and $b_h$ coming into $h$ from the right, such that the initial vertices $i(a_h)$ and $i(b_h)$ are the barycenters of faces of $Y$ and the terminal vertices $t(a_h)$ and $t(b_h)$ are the barycenters of edges of $Y$.  Then for each edge $a$ which is parallel to $a_h$, put $\phi(a) = x_{h}$, and for each edge $b$ which is parallel to $b_h$, put $\phi(b) = 1$.  The remaining edges in $Y'$ are all of the form $c = ab$ for some edges $a$ and $b$ of $Y'$; for these, put $\phi(c) = \phi(a)\phi(b)$.  

Since the tessellation $Y$ has $\f$ faces, Fact \ref{facts:covering theory}(4) above implies that the covering $\Phi$ has $\f$ sheets, and thus $\Gg$ is an index $\f$ subgroup of $\G_0$.

\smallskip
For \eqref{i:g'}, let $Y$ be the tessellation of $S_g$ constructed in Proposition \ref{p: tiling f composite}. 
Then let $Y_2$ be the double cover of $Y$ obtained by replacing the $x$ by $y$ block in Figure \ref{fig:torus-matching} above with a $x$ by $2y$ block.  (Recall that in this case, $x$ is even.)  Consider the symmetry group of $Y_2$ generated by reflecting across all the \emph{even-numbered} vertical black geodesics and all the  \emph{even-numbered} horizontal yellow geodesics. Then, as in \eqref{i:g}, the quotient is a single $\pgon$ $Y_0$, and pulling back the edge labels on $Y_0$ produces a type-consistent  labeling in $Y_2$.

  Let $\Gg$ and $\G_{\p,\val,g'}$ respectively be the lattices which are the fundamental groups of the complexes of groups $G(Y)$ and $G(Y_2)$ as in Theorem \ref{thm:existence}(a) above.  It is straightforward to construct a $2$--sheeted covering $G(Y_2) \to G(Y)$, so that $\G_{\p,\val,g'}$ is an index $2$ subgroup of $\Gg$.  A  covering $\Phi':G(Y_2) \to G(Y_0)$ may then be constructed similar to the covering $\Phi$ in (1).  Since the tessellation $Y_2$ has $2\f$ faces, it follows that $\G_{\p,\val,g'}$ is an index $2F$ subgroup of $\G_0$.

\smallskip
  The proof of \eqref{i:g''} is similar to \eqref{i:g'}. This time, we start by constructing a four-fold cover $Y_4$ of $Y$, by replacing the $x$ by $y$ block in Figure \ref{fig:torus-matching} above by a $2x$ by $2y$ block.  (Recall that in this case, $x$ and $y$ are any positive integers such that $F = xy$.) Then proceed exactly as in \eqref{i:g'}.
\end{proof}

Now suppose that $\val$ is odd.  Then by $\val$--thickness, if a lattice $\Gg = \pi_1(G(Y))$ exists,  the face groups of the complex of groups $G(Y)$ cannot all be trivial.  Hence by Fact \ref{facts:covering theory}(3) above, when $\val$ is odd there are \emph{no} values of $\p$ and $g$ such that a lattice $\Gg$ embeds in $\G_0$.  
However, since the local groups in $G(Y)$ are still just direct products of cyclic groups, it seems possible that the nontrivial face groups could be killed using an ``unfolding" construction similar to that in \cite{ABJLMS09} or \cite{KT08}, so as to obtain an explicit finite index subgroup of $\Gg$ which embeds in $\G_0$. 

Finally, we note that the lattices constructed in Kubena--Thomas \cite{KT08} are explicitly commensurable with $\G_0$, and so in the case $\val$ even Proposition \ref{p:rel with Gamma_0} above provides an explicit commensuration of the lattices in \cite{KT08} with certain $\Gg$.  In order to describe the relationship between the lattices $\Gg$ and the examples in Thomas \cite{T06} or Vdovina \cite{V05}, it would likely also be easiest first to relate those examples directly to $\G_0$.  

\subsection{Relationship with surface subgroups}\label{ss:surface subgroups}

Since each $\Gg$ is the fundamental group of a \emph{simple} complex of groups, the (topological) fundamental group $\pi_1(S_g)$ of the quotient genus $g$ surface $S_g = \Gg \backslash \I$ embeds in $\Gg$.  We now discuss this embedding, and also draw some conclusions about how $\pi_1(S_g)$ sits inside $\Aut(\I)$.

Let $\Gg = \pi_1(G(Y))$.   We first sketch an algebraic argument that the embedding of $\pi_1(S_g)$ in $\Gg$ is of infinite index.   Consider the presentation of $\pi_1(G(Y))$ given in Section \ref{ss:complexes_of_groups} above.  The surface subgroup $\pi_1(S_g)$ is generated by $E^{\pm}(Y')$, and since $G(Y)$ is simple, the only relation involving both a nontrivial element of a local group and an element of $E^\pm(Y')$ is (4).  From this it follows that there is some nontrivial element $h$ of a local group $G_\s$, and some edge $a$ of $Y'$ with $\sigma \neq i(a)$ and $a^+$ nontrivial in $\pi_1(S_g)$, so that $h$ and $a^+$ together generate an infinite subgroup $\langle h,a^+ \rangle$ of $\Gg$.  Since $h$ is not in $\pi_1(S_g)$ and $a^+$ generates an infinite cyclic group, there will then be infinitely many distinct $\langle h, a^+ \rangle$--cosets of $\pi_1(S_g)$ in $\Gg$.  Hence $\pi_1(S_g)$ is a subgroup of infinite index.

We now apply some covering theory.  
The surface group $\pi_1(S_g)$  is itself the fundamental group of a complex of groups $G(Y_g)$ with $Y_g = \pi_1(S_g) \backslash \I$.   Hence there is a covering of complexes of groups $\Phi: G(Y_g) \to G(Y)$, with infinitely many sheets.   Recall that by construction, at least one face group in $G(Y)$ is trivial.  Thus by Fact \ref{facts:covering theory}(5) above, the polygonal complex $Y_g$ is infinite.  So in particular, $\pi_1(S_g) = \pi_1(G(Y_g))$ does not act cocompactly on $\I$.  Also, all local groups of $G(Y)$ are finite, hence by Fact \ref{facts:covering theory}(3), all local groups of $G(Y_g)$ are finite.  Thus $\pi_1(S_g)$ is a discrete subgroup of $\Aut(\I)$.  But since $\pi_1(S_g)$ is torsion-free, $G(Y_g)$ has all local groups trivial.  Thus $\pi_1(S_g)$ is a discrete subgroup of $\Aut(\I)$ which acts freely but not cocompactly on $\I$.  Hence in particular, $\pi_1(S_g)$ is not a uniform lattice in $\Aut(\I)$.  Moreover, by the characterization of lattices in Section \ref{ss:lattices} above, a nonuniform lattice must have torsion, in order for its $S$--covolume to converge.  Therefore $\pi_1(S_g)$ is not a lattice in $\Aut(\I)$.  

We finally note that since all local groups of $G(Y)$ are direct products of cyclic groups, it seems possible that all of its nontrivial local groups (not just the nontrivial face groups) could be killed using a construction similar to that in \cite{ABJLMS09} or \cite{KT08}, so as to obtain the infinite polygonal complex $Y_g = \pi_1(S_g) \backslash \I$ explicitly.

\bibliographystyle{hamsplain}
\bibliography{biblio}

\def\cprime{$'$}
\providecommand{\bysame}{\leavevmode\hbox to3em{\hrulefill}\thinspace}
\providecommand{\href}[2]{#2}
\begin{thebibliography}{10}

\bibitem{ABJLMS09}
Goulnara Arzhantseva, Martin~R. Bridson, Tadeusz Januszkiewicz, Ian~J. Leary,
  Ashot Minasyan, and Jacek {\'S}wi{\c{a}}tkowski, \emph{Infinite groups with
  fixed point properties}, Geom. Topol. \textbf{13} (2009), no.~3, 1229--1263.

\bibitem{BK90}
Hyman Bass and Ravi Kulkarni, \emph{Uniform tree lattices}, J. Amer. Math. Soc.
  \textbf{3} (1990), no.~4, 843--902.

\bibitem{BL01}
Hyman Bass and Alexander Lubotzky, \emph{Tree lattices}, Progress in
  Mathematics, vol. 176, Birkh\"auser Boston Inc., Boston, MA, 2001, With
  appendices by Bass, L. Carbone, Lubotzky, G. Rosenberg and J. Tits.

\bibitem{B97}
Marc Bourdon, \emph{Immeubles hyperboliques, dimension conforme et rigidit\'e
  de {M}ostow}, Geom. Funct. Anal. \textbf{7} (1997), no.~2, 245--268.

\bibitem{BP00}
Marc Bourdon and Herv{\'e} Pajot, \emph{Rigidity of quasi-isometries for some
  hyperbolic buildings}, Comment. Math. Helv. \textbf{75} (2000), no.~4,
  701--736.

\bibitem{BH99}
Martin~R. Bridson and Andr{\'e} Haefliger, \emph{Metric spaces of non-positive
  curvature}, Grundlehren der Mathematischen Wissenschaften [Fundamental
  Principles of Mathematical Sciences], vol. 319, Springer-Verlag, Berlin,
  1999.

\bibitem{C92}
Jon~Michael Corson, \emph{Complexes of groups}, Proc. London Math. Soc. (3)
  \textbf{65} (1992), no.~1, 199--224.

\bibitem{EEK}
Allan~L. Edmonds, John~H. Ewing, and Ravi~S. Kulkarni, \emph{Regular
  tessellations of surfaces and {$(p,\,q,\,2)$}-triangle groups}, Ann. of Math.
  (2) \textbf{116} (1982), no.~1, 113--132.

\bibitem{H02}
Fr{\'e}d{\'e}ric Haglund, \emph{Existence, unicit\'e et homog\'en\'eit\'e de
  certains immeubles hyperboliques}, Math. Z. \textbf{242} (2002), no.~1,
  97--148.

\bibitem{H06}
\bysame, \emph{Commensurability and separability of quasiconvex subgroups},
  Algebr. Geom. Topol. \textbf{6} (2006), 949--1024 (electronic).

\bibitem{H08}
\bysame, \emph{Finite index subgroups of graph products}, Geom. Dedicata
  \textbf{135} (2008), 167--209.

\bibitem{hooley}
Christopher Hooley, \emph{On {A}rtin's conjecture}, J. Reine Angew. Math.
  \textbf{225} (1967), 209--220.

\bibitem{J09}
Martin Jones, \emph{{Groups Acting on a Product of Two Trees}}, Ph.D. thesis,
  University of Newcastle-upon-Tyne, 2009.

\bibitem{K07}
Sang-hyun Kim, \emph{{Hyperbolic Surface Subgroups of Right-Angled Artin Groups
  and Graph Products of Groups}}, Ph.D. thesis, Yale University, 2007.

\bibitem{KT08}
Angela Kubena and Anne Thomas, \emph{Density of commensurators for uniform
  lattices of right-angled buildings}, \mbox{arXiv:0812.2280}.

\bibitem{LL09}
Fran\c{c}ois Ledrappier and Seonhee Lim, \emph{Volume entropy of hyperbolic
  buildings}, J. Mod. Dyn. \textbf{4} (2010), 139--165.

\bibitem{LT07}
Seonhee Lim and Anne Thomas, \emph{Covering theory for complexes of groups}, J.
  Pure Appl. Algebra \textbf{212} (2008), no.~7, 1632--1663.

\bibitem{R02}
Bertrand R{\'e}my, \emph{Immeubles de {K}ac-{M}oody hyperboliques, groupes non
  isomorphes de m\^eme immeuble}, Geom. Dedicata \textbf{90} (2002), 29--44.

\bibitem{S91}
John~R. Stallings, \emph{Non-positively curved triangles of groups}, Group
  theory from a geometrical viewpoint ({T}rieste, 1990), World Sci. Publ.,
  River Edge, NJ, 1991, pp.~491--503.

\bibitem{T06}
Anne Thomas, \emph{Lattices acting on right-angled buildings}, Algebr. Geom.
  Topol. \textbf{6} (2006), 1215--1238 (electronic).

\bibitem{V05}
Alina Vdovina, \emph{Groups, periodic planes and hyperbolic buildings}, J.
  Group Theory \textbf{8} (2005), no.~6, 755--765.

\bibitem{Z97}
A.~Zvonkin, \emph{Matrix integrals and map enumeration: an accessible
  introduction}, Math. Comput. Modelling \textbf{26} (1997), no.~8-10,
  281--304, Combinatorics and physics (Marseilles, 1995).

\end{thebibliography}

\end{document}